\documentclass[12pt,reqno,twoside]{amsart}
\usepackage{amsthm,amsmath,amssymb,mathrsfs,amsbsy}
\usepackage{graphicx,epsfig,wrapfig,sidecap,subcaption}
\usepackage[all]{xy}
\usepackage{proof}
\usepackage[usenames]{color}
\usepackage{hyperref}
\usepackage{bm}
\usepackage[capitalize]{cleveref}
\usepackage{stmaryrd}
\usepackage{tikz}
\usepackage{mathdots}

\xyoption{dvips}
\xyoption{color}
\usepackage{a4wide} 	

\usepackage{color}

\title[Spiral Delone Sets]{Higher Dimensional Spiral Delone Sets}
\author{Faustin Adiceam and Ioannis Tsokanos}
\date{}

\newcommand{\N}{{\mathbb{N}}}
\newcommand{\Z}{{\mathbb{Z}}}

\newcommand{\R}{{\mathbb{R}}}

\newcommand{\T}{{\mathbb{T}}}

\newcommand{\Sph}{\mathbb{S}}
\newcommand{\K}{\mathbb{K}}
\newcommand{\dsph}{d_{\mathbb{S}^d}}
\newcommand{\dsphb}{d_{\mathbb{S}^2}}

\theoremstyle{plain}
\newtheorem{thm}{Theorem}[section]

\newtheorem{prop}[thm]{Proposition}

\theoremstyle{definition}

\numberwithin{equation}{section}
\swapnumbers

\usepackage{color}

\usepackage{tikz-3dplot}
\tdplotsetmaincoords{80}{110}

\begin{document}

\begin{abstract}
A Delone set in $\R^n$ is a set such that (a) the distance between any two of its points is uniformly bounded below by a strictly positive constant and such that (b) the distance from any point in $\R^n$ to the set is uniformly bounded above. 
Delone sets are thus sets of points enjoying nice spacing properties, and appear therefore naturally in mathematical models for quasicrystals.

Define a spiral set in $\R^n$ as a set of points of the form $\left\{\sqrt[n]{k}\cdot\bm{u}_k\right\}_{k\ge 1}$, where $\left(\bm{u}_k\right)_{k\ge 1}$ is a sequence in the unit sphere $\Sph^{n-1}$. 
In the planar case $n=2$, spiral sets serve as natural theoretical models in phyllotaxis (the study of  configurations of leaves on a plant stem), and an important example in this class includes the sunflower spiral.

Recent works by  Akiyama, Marklof and Yudin provide a reasonable complete characterisation of planar spiral sets which are also Delone. A related problem that has emerged in several places in the literature over the past fews years is to determine whether this theory can be extended to higher dimensions, and in particular to show the existence of spiral Delone sets in any dimension.

This paper addresses this question by characterising the Delone property of a spiral set in terms of packing and covering conditions satisfied by the spherical sequence $\left(\bm{u}_k\right)_{k\ge 1}$. This allows for the construction of explicit examples of spiral Delone sets in $\R^n$ for all $n\ge 2$, which boils down to finding a sequence of points in $\Sph^{n-1}$ enjoying some optimal  distribution properties.
\end{abstract}

\maketitle

\begin{center}
\emph{A Monsieur Fran\c{c}ois Nebout.}
\end{center}



\section{Introduction}\label{sec:introduction} 

Throughout, the integers $n\ge 2$ and $d\ge 1$ denote two dimensions related as follows~: 
\begin{equation}\label{elem-1}
n=d+1.
\end{equation}
In particular, the unit sphere in $\R^n$ will be denoted by $\Sph^d$. A set $\mathfrak{B}$ in $\R^n$ is said to be \emph{uniformly discrete} if $$\inf_{\underset{\bm{x}\neq \bm{y}}{\bm{x}, \bm{y}\in\mathfrak{B}}}\; \left\|\bm{x}-\bm{y}\right\|_2\;>\;0;$$ that is, if the distance between two distinct points in $\mathfrak{B}$ is uniformly bounded below by a strictly positive constant. It is \emph{relatively dense} if $$\sup_{\bm{x}\in\R^n}\; \inf_{\bm{y}\in\mathfrak{B}} \; \left\|\bm{x}-\bm{y}\right\|_2\;<\;\infty;$$  that is, if the distance between a point 
and its complement in $\mathfrak{B}$ is uniformly bounded above. The set $\mathfrak{B}$ is said to be a \emph{Delone set} if it is both uniformly discrete and relatively dense. A \emph{spiral set} in $\R^n$ is a set of points of the form $$\mathfrak{S}=\left\{f(k)\bm{u}_k\right\}_{k\ge 1},$$ where $f~: \R_+\rightarrow \R_+$ is a strictly increasing sequence, and where $\left(\bm{u}_k\right)_{k\ge 1}$ is a sequence in $\Sph^d$.\\

The goal of  this paper is to relate the Delone property and spiral sets. This problem is reasonable well--understood in dimension $n=2$, where interests in spiral sets are partly motivated by the fact that they can be used as models for phyllotaxis configurations (i.e.~for configurations of leaves on a plant stem). In this case, it follows from the Lie group structure of the abelian group $\Sph^1$ that the spherical sequence $\left(\bm{u}_k\right)_{k\ge 1}$ can be decomposed as $\bm{u}_k=e\!\left(x_k\right)$, where $e(x)=\exp(2i\pi x)$ for all $x\in\R$, and where $\left(x_k\right)_{k\ge 1}$ is a real sequence.  Akiyama~\cite[Lemma~1]{akiyama} showed that a necessary condition for a planar spiral set $\mathfrak{S}$ to be Delone is that $$0\;<\; \liminf_{k\rightarrow\infty}\;\frac{f(k)}{\sqrt{k}} \;\le\; \limsup_{k\rightarrow\infty}\;\frac{f(k)}{\sqrt{k}}\;<\;\infty.$$ This suggests that a very natural choice for the function $f$ is $f(x)=\sqrt{x}$. His proof can be generalised in an obvious way to any dimension so that a natural choice for $f$ in dimension $n\ge 2$ is $f(x)=\sqrt[n]{x}$.

Akiyama applied his result to study  Fermat's spiral  $\mathfrak{S}_{\alpha}=\left\{\sqrt{k}\cdot e(k\alpha)\right\}_{k\ge 1}$, where $\alpha\in\R$ (this is also known as the sunflower spiral). With the help of the Three Distance Theorem, he proved in~\cite[Theorem~3]{akiyama} that $\mathfrak{S}_{\alpha}$ is Delone if, and only if, $\alpha$ is badly approximable number; that is, if, and only if, $$\inf_{\underset{(p,q)\neq(0,0)}{(p,q)\in\Z^2}}\left\{q\left|q\alpha-p\right|\right\}\;>\;0.$$ Akiyama thus rediscovered a result already noticed by Yudin~\cite{yudin}, who used analytic tools to establish the sufficiency of the property of bad approximability in this equivalence.  Of particular interest in considerations related to phyllotaxis  is the case when the badly approximable number $\alpha$ is the Golden Ratio $\varphi=(1+\sqrt{5})/2$. A portion of the resulting spiral Delone set $\mathfrak{S}_{\varphi}$ is represented in Figure~\ref{sunflower}.\\

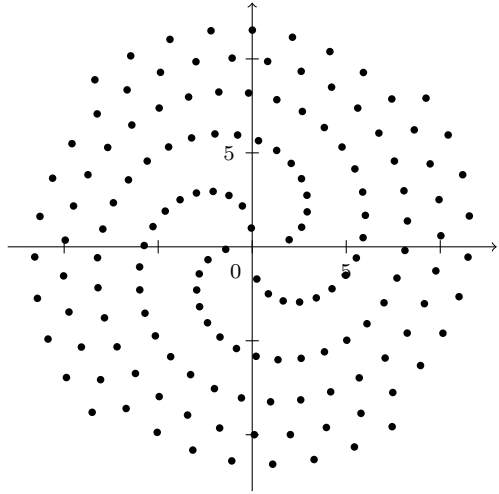
\begin{figure}[h!]
\centering
\begin{tikzpicture}[scale=0.25]
\fill (-0.047220096, 0.998884509) circle (0.2) ;
\fill (-1.407906912, -0.133409618) circle (0.2) ;
\fill (0.244633358, -1.714687878) circle (0.2) ;
\fill (1.964403748, 0.375656646) circle (0.2) ;
\fill (-0.523236484, 2.173987944) circle (0.2) ;
\fill (-2.351762483, -0.684991403) circle (0.2) ;
\fill (0.858998149, -2.502423262) circle (0.2) ;
\fill (2.628856255, 1.043606627) circle (0.2) ;
\fill (-1.237342121, 2.732944287) circle (0.2) ;
\fill (-2.81597484, -1.438848741) circle (0.2) ;
\fill (1.646854674, -2.878866041) circle (0.2) ;
\fill (2.922300856, 1.860149916) circle (0.2) ;
\fill (-2.077571636, 2.946811174) circle (0.2) ;
\fill (-2.952832817, -2.29799442) circle (0.2) ;
\fill (2.520323763, -2.940742786) circle (0.2) ;
\fill (2.910885208, 2.743491809) circle (0.2) ;
\fill (-2.966454668, 2.863589827) circle (0.2) ;
\fill (-2.799185409, -3.188190874) circle (0.2) ;
\fill (3.407700677, -2.718009583) circle (0.2) ;
\fill (2.620416132, 3.624005973) circle (0.2) ;
\fill (-3.836150681, 2.506780396) circle (0.2) ;
\fill (-2.377503255, -4.043201488) circle (0.2) ;
\fill (4.244248853, -2.233014033) circle (0.2) ;
\fill (2.073772547, 4.438408208) circle (0.2) ;
\fill (-4.624821309, 1.900270469) circle (0.2) ;
\fill (-1.713032155, -4.802657685) circle (0.2) ;
\fill (4.97111616, -1.512614995) circle (0.2) ;
\fill (1.299609407, 5.129426419) circle (0.2) ;
\fill (-5.276850588, 1.074638486) circle (0.2) ;
\fill (-0.838357388, -5.412684813) circle (0.2) ;
\fill (5.53626083, -0.591452464) circle (0.2) ;
\fill (0.334640189, 5.646947489) circle (0.2) ;
\fill (-5.744152243, 0.068665893) circle (0.2) ;
\fill (0.205697667, -5.827322582) circle (0.2) ;
\fill (5.895947391, 0.487651886) circle (0.2) ;
\fill (-0.776374031, 5.94955825) circle (0.2) ;
\fill (-5.987730634, -1.071019073) circle (0.2) ;
\fill (1.370721633, -6.010085041) circle (0.2) ;
\fill (6.016288014, 1.674598021) circle (0.2) ;
\fill (-1.981748366, 6.006053064) circle (0.2) ;
\fill (-5.979141496, -2.291258818) circle (0.2) ;
\fill (2.602203818, -5.935363114) circle (0.2) ;
\fill (5.87457682, 2.913648433) circle (0.2) ;
\fill (-3.224650729, 5.796691097) circle (0.2) ;
\fill (-5.70166437	, -3.534264197) circle (0.2) ;
\fill (3.841540204, -5.589505243) circle (0.2) ;
\fill (5.460272619, 4.145530476) circle (0.2) ;
\fill (-4.44528959	, 5.314075692) circle (0.2) ;
\fill (-5.151073809, -4.739877489) circle (0.2) ;
\fill (5.028361984, -4.971476215) circle (0.2) ;
\fill (4.775541661, 5.309821263) circle (0.2) ;
\fill (-5.583346383, 4.56357789) circle (0.2) ;
\fill (-4.335941002, -5.848043744) circle (0.2) ;
\fill (6.103037532, -4.093034679) circle (0.2) ;
\fill (3.835309306, 6.347472137) circle (0.2) ;
\fill (-6.580514524, 3.563260951) circle (0.2) ;
\fill (-3.277430237, -6.801356559) circle (0.2) ;
\fill (7.009217285, -2.978401089) circle (0.2) ;
\fill (2.666799373, 7.203345133) circle (0.2) ;
\fill (-7.383020071, 2.34329141) circle (0.2) ;
\fill (-2.008582395, -7.547555681) circle (0.2) ;
\fill (7.696301159, -1.6634147) circle (0.2) ;
\fill (1.308566082, 7.828643229) circle (0.2) ;
\fill (-7.944007972, 0.944847787) circle (0.2) ;
\fill (-0.573102565, -8.041862561) circle (0.2) ;
\fill (8.121716897, -0.194202598) circle (0.2) ;
\fill (-0.190952664, 8.183125141) circle (0.2) ;
\fill (-8.225687144, -0.581438735) circle (0.2) ;
\fill (0.976308403, -8.249049758) circle (0.2) ;
\fill (8.252908035, 1.374594111) circle (0.2) ;
\fill (-1.775310396, 8.237006313) circle (0.2) ;
\fill (-8.201139171, -2.177456382) circle (0.2) ;
\fill (2.580018325, -8.145152267) circle (0.2) ;
\fill (8.068943039, 2.981972206) circle (0.2) ;
\fill (-3.382286353, 7.972461291) circle (0.2) ;
\fill (-7.855709631, -3.779924099) circle (0.2) ;
\fill (4.173846465, -7.71874379) circle (0.2) ;
\fill (7.561672795, 4.563014853) circle (0.2) ;
\fill (-4.946393761, 7.384659015) circle (0.2) ;
\fill (-7.187918067, -5.322953491) circle (0.2) ;
\fill (5.691672863, -6.971718584) circle (0.2) ;
\fill (6.736381854, 6.051541912) circle (0.2) ;
\fill (-6.401564575, 6.482281311) circle (0.2) ;
\fill (-6.209841904, -6.740761346) circle (0.2) ;
\fill (7.068171914, -5.919539323) circle (0.2) ;
\fill (5.611899095, 7.382857749) circle (0.2) ;
\fill (-7.683904658, 5.287495552) circle (0.2) ;
\fill (-4.946950664, -7.970425279) circle (0.2) ;
\fill (8.241561531, -4.590932752) circle (0.2) ;
\fill (4.220155076, 8.496486988) circle (0.2) ;
\fill (-8.734409196, 3.835374297) circle (0.2) ;
\fill (-3.437388833, -8.954571906) circle (0.2) ;
\fill (9.15625723, -3.027037088) circle (0.2) ;
\fill (2.605195583, 9.338787714) circle (0.2) ;
\fill (-9.501528309, 2.172776978) circle (0.2) ;
\fill (-1.73072799	, -9.643888252) circle (0.2) ;
\fill (9.765322848, -1.280027219) circle (0.2) ;
\fill (0.821682884, 9.865335131) circle (0.2) ;
\fill (-9.943477428, 0.356730472) circle (0.2) ;
\fill (0.113769697, -9.999352802) circle (0.2) ;
\fill (10.0326163654435, 0.588734969095079) circle (0.2) ;
\fill (-1.06706293529127, 10.0429764856903) circle (0.2) ;
\fill (-10.0301958508321, -1.54763406332044) circle (0.2) ;
\fill (2.0293143904353, -9.99409241025778) circle (0.2) ;
\fill (9.93454018050377, 2.51095826368263) circle (0.2) ;
\fill (-2.99141112319413, 9.85146991529845) circle (0.2) ;
\fill (-9.74486963746393, -3.46951232147021) circle (0.2) ;
\fill (3.94409797180841, -9.61478503081461) circle (0.2) ;
\fill (9.46131969052783, 4.41400381894152) circle (0.2) ;
\fill (-4.87806812487338, 9.2846352308044) circle (0.2) ;
\fill (-9.08495124898736, -5.33513456283185) circle (0.2) ;
\fill (5.78405511222333, -8.86254514565445) circle (0.2) ;
\fill (8.61775180055148, 6.22369294744622) circle (0.2) ;
\fill (-6.65292531340575, 8.35096310459129) circle (0.2) ;
\fill (-8.0626273484935, -7.07064638058818) circle (0.2) ;
\fill (7.47577007257865, -7.75324846899269) circle (0.2) ;
\fill (7.42338515390055, 7.86723285894405) circle (0.2) ;
\fill (-8.24399650647049, 7.07364980765251) circle (0.2) ;
\fill (-6.70470737931605, -8.60505078182255) circle (0.2) ;
\fill (8.94941609878173, -6.31727405538705) circle (0.2) ;
\fill (5.91211582003376, 9.27614610333982) circle (0.2) ;
\fill (-9.58433019005179, 5.49004688577992) circle (0.2) ;
\fill (-5.05192799795299, -9.87309594319323) circle (0.2) ;
\fill (10.1416114964359, -4.59866461653378) circle (0.2) ;
\fill (4.13120497936299, 10.3890878049272) circle (0.2) ;
\fill (-10.6147808238573, 3.65053805095509) circle (0.2) ;
\fill (-3.15769136147347, -10.8179935877997) circle (0.2) ;
\fill (10.9980781853378, -2.65372874069619) circle (0.2) ;
\fill (2.13974795207348, 11.1544376237261) circle (0.2) ;
\fill (-11.2865275785817, 1.61687823224722) circle (0.2) ;
\fill (-1.0862777416431, -11.3938580238658) circle (0.2) ;
\fill (11.4759947376923, -0.549130931981181) circle (0.2) ;
\fill (0.00664583678150759, 11.532560679782) circle (0.2) ;
\fill (-11.5632372366825, -0.53994870885903) circle (0.2) ;
\fill (1.08940591274941, -11.5677653311807) circle (0.2) ;
\fill (11.54594639265, 1.64046392768979) circle (0.2) ;
\fill (-2.19184880435134, 11.4976431854039) circle (0.2) ;
\fill (-11.4227804924582, -2.74227748798641) circle (0.2) ;
\fill (3.29046085188387, -11.3213456524487) circle (0.2) ;
\fill (11.1933889477963, 3.83510676035917) circle (0.2) ;
\fill (-4.37492315398156, 11.0390238425667) circle (0.2) ;
\fill (-10.8584270688299, -4.90862114966132) circle (0.2) ;
\fill (5.4349181481498, -10.6518385606858) circle (0.2) ;
\fill (10.41956123549, 5.95254094146984) circle (0.2) ;
\fill (-6.46022881276794, 10.1619606221773) circle (0.2) ;
\fill (-9.87946433695444, -6.95673662106345) circle (0.2) ;
\fill (7.44083786338859, -9.5725614069988) circle (0.2) ;
\fill (9.24180144316953, 7.91132770684097) circle (0.2) ;
\fill (-8.36702598310921, 8.88779366310759) circle (0.2) ;
\fill (-8.51120576646421, -8.80678013810419) circle (0.2) ;

\draw[thin,->] (0,-13) -- (0,13);
\draw[thin,->] (-13,0) -- (13,0);

\foreach \z in {0}
   \draw (\z cm,10pt) -- (\z cm,-10pt) node[anchor=north east]   {\tiny{$\z$}};
\foreach \x in {5}
   \draw (\x cm,10pt) -- (\x cm,-10pt) node[anchor=north] {\tiny{$\x$}};
\foreach \x in {-10, -5, 10}
   \draw (\x cm,10pt) -- (\x cm,-10pt) node[anchor=north] {};
\foreach \y in {5}
   \draw (10pt, \y cm) -- (-10pt, \y cm) node[anchor=east] {\tiny{$\y$}};
\foreach \y in {-10, -5, 10}
   \draw (10pt, \y cm) -- (-10pt, \y cm) node[anchor=east] {};      
\end{tikzpicture}
\caption{First 150 points on Fermat's spiral $\left\{\sqrt{n}\cdot e(n\varphi)\right\}_{ n\ge 1}$, where $\varphi=(1+\sqrt{5})/2$ is the Golden Ratio.}
\label{sunflower}
\end{figure}

Marklof~\cite{marklof} then undertook a more general study of planar spiral sets $\mathfrak{S}_{\bm{x}}=\left\{\sqrt{k}\cdot e(x_k)\right\}_{k\ge 1}$ by providing a necessary and sufficient condition on the sequence $\bm{x}=\left(x_k\right)_{k\ge 1}$ for $\mathfrak{S}_{\bm{x}}$ to be Delone\footnote{At this level of generality, it should be mentioned that the terminology ``spiral'' sets, inspired by the case of the sunflower spiral (see Figure~\ref{sunflower}), can be slightly misleading~: the beatiful pictures in~\cite{marklof} illustrate that for various choices of the real sequence $\bm{x}$, the set $\mathfrak{S}_{\bm{x}}$ need not resemble a spiral.}. To state his result, given parameters $h>0$ and $R>0$, let $g_R^h$ (resp.~$G_R^h$) \sloppy denote the minimal gap (resp., the maximal gap) between the fractional parts of the set $\left\{x_k\,:\, R^2\le k< \left(R+h\right)^2\right\}$. Then, as established in~\cite[Proposition 2]{marklof}, $\mathfrak{S}_{\bm{x}}$ is Delone if, and only if, there exist $h, h'>0$ such that 
\begin{equation}\label{marspac}
\inf_{R\ge 1}\; R\cdot g_R^h\;>\;0\qquad\textrm{ and }\qquad \sup_{R\ge 1} R\cdot G_R^{h'}\;<\;\infty.
\end{equation} 
Thus, the Delone property of the set $\mathfrak{S}_{\bm{x}}$  is closely related to the spacing of the sequence $\bm{x}$ modulo one. As proved in~\cite[Proposition 3]{marklof}, \eqref{marspac} implies in particular that the set $\left\{\sqrt{k}\cdot e(\alpha\sqrt{k})\right\}_{k\ge 1}$ is Delone for \emph{any} $\alpha\neq 0$.\\

A question raised by Akiyama at the end of~\cite{akiyama} is to determine whether this planar theory can be extended to higher dimensions, and in particular how to obtain higher dimensional spiral Delone sets. This question also appears in~\cite{mthfflow}, where the problem of generalising the sunflower spiral to higher dimensions is left open.  In view of the above discussion, we will from now on restrict ourselves to the case where the spiral set 
\begin{equation}\label{sk}
\widetilde{\mathfrak{S}}\;=\;\left\{\bm{s}_k\right\}_{k\ge 1}
\end{equation} 
is defined in $\R^n$ by
\begin{equation*}
\bm{s}_k\;=\;\sqrt[n]{k}\cdot \bm{u}_k
\end{equation*} 
for all $k\ge 1$. Here,  $\left(\bm{u}_k\right)_{k\ge 1}$ denotes again a spherical sequence. 

As suggested by the conditions established in the planar case by Marklof, and as further explicited in Section~\ref{seccns} below, the Delone property of $\widetilde{\mathfrak{S}}$ strongly relies on the distribution properties of the spherical sequence $\left(\bm{u}_k\right)_{k\ge 1}$. This is, of course, closely related to the well--known problem of ``evenly'' distributing  points on the unit sphere $\Sph^d$, although the precise properties that the sequence $\left(\bm{u}_k\right)_{k\ge 1}$ must satisfy in our case are  different from equidistribution. The main difference with the planar case is that when $d\ge 2$, the sphere $\Sph^d$ does not enjoy an abelian group structure, and therefore cannot be identified any more with the torus $\T^d=\R^d\backslash\Z^d$. 

The main result of this paper reads as follows~:

\begin{thm}\label{mainthm}
Let $n\ge 2$. Then, there exists an effectively constructible sequence  $\left(\bm{u}_k\right)_{k\ge 1}$ in $\Sph^d$ such that the spiral set $\widetilde{\mathfrak{S}}$ defined in~\eqref{sk} is Delone.
\end{thm}

Akiyama~\cite{akiyama} suggests that the sequence  $\left(\bm{u}_k\right)_{k\ge 1}$ might be chosen as the orbit of a transitive map, namely that $\bm{u}_k=T^k(\bm{v})$ for all $k\ge 1$ with $\bm{v}\in\Sph^d$ and $T~: \Sph^d\rightarrow\Sph^d$ transitive. We adopt here a rather different approach in order to define our spherical sequence. 

In Section~\ref{seccns}, we first provide necessary and sufficient conditions on the spherical sequence $\left(\bm{u}_k\right)_{k\ge 1}$  for the spiral set~\eqref{sk} to be uniformly discrete and relatively dense. These conditions are expressed in terms of packing and covering properties of the sequence  $\left(\bm{u}_k\right)_{k\ge 1}$ and provide a natural generalisation of the planar case. 

The main idea in the construction of a sequence $\left(\bm{u}_k\right)_{k\ge 1}$ satisfying these conditions is then presented in Section~\ref{liftsection}~: it relies on a process of lifting  (Proposition~\ref{lifting}) which maps a toral sequence to the sphere by preserving the required packing and covering properties. As detailed in Section~\ref{liftsection}, this lifting technique also constitutes a generalisation of the planar case.

Section~\ref{secdispersion} completes the construction of the spherical sequence by defining it from the lifting of a linear toral flow. The main result in this section (Proposition~\ref{propdensbad}) provides a sharp, quantitative estimate on the rate of dispersion of the multiples of a badly approximable vector modulo one.

Section~\ref{sec3} focuses on the three dimensional case by providing an alternative to the lifting argument introduced in Section~\ref{liftsection}. A spiral Delone set is defined there from a spherical sequence obtained from a continuous spherical flow (which results from the radial projection of a flow lying on the surface of a tetrahedron). 

Section~\ref{oppb} concludes the paper with a collection of open problems emerging from the theories elaborated within.\\

\paragraph{\textbf{Notation and background}} If $x,y>0$, the Vinogradov notation $x\ll y$ and $x\gg y$ implies the existence of constants $\gamma, \gamma'>0$ independent of $x$ and $y$ such that $x\le \gamma y$ and $x\ge \gamma' y$, respectively. Also, $x\asymp y$ means that the two relations $x\ll y$ and $x\gg y$ hold simultaneously.

Given $\bm{u}, \bm{v}\in\Sph^d$, the geodesic length between $\bm{u}$ and $\bm{v}$ is denoted by $\dsph(\bm{u}, \bm{v})$; that is, $$\dsph(\bm{u}, \bm{v})\;=\; \arccos(\bm{u\cdot v}),$$ where $\bm{x\cdot y}$ stands for the usual scalar product between two vectors $\bm{x}, \bm{y}\in\R^n$. It is easily seen that this geodesic distance is equivalent to the one induced by the Euclidean norm $\left\|\, .\, \right\|_2$ in the following sense~:  for all $\bm{u}, \bm{v}\in\Sph^d$, 
\begin{equation}\label{elem0}
\dsph(\bm{u}, \bm{v})\; \asymp \; \left\|\, \bm{u}- \bm{v}\, \right\|_2.
\end{equation}
Also, $\sigma_d$ denotes the measure on $\Sph^d$ induced by the $n$--dimensional Lebesgue measure $\lambda_n$. Given $\bm{w}\in\Sph^d$ and $\rho>0$, $\mathcal{C}_d\left(\bm{w}, \rho\right)$ stands for the spherical cap centered at $\bm{w}$ with geodesic length $\rho$; that is, $$\mathcal{C}_d\left(\bm{w}, \rho\right)\;=\; \left\{\bm{u}\in\Sph^d\;:\; \dsph\left(\bm{w}, \bm{u}\right)\le \rho\right\}.$$

Throughout, set $$\K_d:=\left[0,1\right)^d\qquad \textrm{and}\qquad \K'_d=[-1,1)^d.$$ It will be convenient whenever needed to see $\K_d$ as the unit torus $\T^d=\R^d\backslash\Z^d$ upon suitably identifying its edges, and $ \K'_d$ as the $2\times2$ torus.

Given a vector $\bm{x}\in\R^d$, $\left\{\bm{x}\right\}$ stands for the vector each of whose components are the  fractional parts of the corresponding components of $\bm{x}$. The sup norm of $\bm{x}$ is denoted by $\left\|\bm{x}\right\|_{\infty}$ and we set $$\left\|\bm{x}\right\|_{\Z^d }\; =\; \min_{\bm{m}\in\Z^d}\; \left\|\bm{x}-\bm{m}\right\|_{\infty}.$$ When $d=1$, the notation $\left\|\, . \, \right\|_{\Z^d }$ is simplified to $\left\|\, . \, \right\|$. Note that, for any $\bm{x}, \bm{y}\in\R^d$, $\left\| \bm{x} - \bm{y}\right\|_{\Z^d }$ is the toral distance between the projections of $\bm{x}$ and $\bm{y}$ onto $\T^d$, and that this distance satisfies the following relation which will often be used implicitly~:  $$\left\| \bm{x} - \bm{y}\right\|_{\Z^d }\;\le\; \left\| \bm{x}- \bm{y} \right\|_{\infty}.$$ 

We will also occasionally use the notation $\bm{0}$ and $\bm{1}$ to denote a vector all of whose components are equal to 0 and 1, respectively. The dimension of these vectors will be clear from the context.

Given a subset $\mathcal{D}\subset\R^d$, its topological interior, closure and boundary are denoted by $\mathcal{D}^{\circ}, \overline{\mathcal{D}}$ and $\partial \mathcal{D}$, respectively (in particular, $\partial \mathcal{D}= \overline{\mathcal{D}}\backslash \mathcal{D}^{\circ}$).

Finally, the cardinality of a finite set $A$ is denoted by $\# A$.\\

\paragraph{\textbf{Acknowledgments}} The first--named author's work was supported in part by EPSRC Grant EP/T021225/1. He would like to thank Daniel El--Baz for drawing his attention to the reference~\cite{plato}.

\section{On the Delone Property of Spiral Sets}\label{seccns}

Let $\left(\bm{u}_k\right)_{k\ge 1}$ be a sequence lying in $\Sph^d$. The goal in this section is to provide necessary and sufficient conditions for the spiral sequence $\widetilde{\mathfrak{S}}$ defined in~\eqref{sk} to be Delone.

The property of relative density of this set is addressed in the following proposition~:

\begin{prop}[Necessary and sufficient condition for relative density]\label{cnsreldens} The set $\widetilde{\mathfrak{S}}$ defined in~\eqref{sk} is relatively dense if, and only if, there exist two constants $c, C>0$ such that for all $k\ge 1$ and all $\bm{v}\in\Sph^d$, there exists an integer $m$ such that 
\begin{equation}\label{reldens}
 \left|m\right|\le ck^{1-1/n}\qquad \textrm{and}\qquad \dsph\left(\bm{u}_{k+m}, \bm{v}\right)\le\frac{C}{\sqrt[n]{k}}\cdotp
\end{equation} 
\end{prop}
 
Condition~\eqref{reldens} can be rephrased in terms of a covering property of the sequence $\left(\bm{u}_k\right)_{k\ge 1}$~: it is saying that for any $k\ge 1$, the spherical caps with radius $Ck^{-1/n}$ centered at $\bm{u}_{k+m}$ cover the sphere when $\left|m\right|\le ck^{1-1/n}$; that is, that $$\bigcup_{\left|m\right|\le ck^{1-1/n}} \mathcal{C}_d\left(\bm{u}_{k+m}, \frac{C}{\sqrt[n]{k}}\right)\;=\; \Sph^d$$ for all $k\ge 1$.

The proof of Proposition~\ref{cnsreldens} relies on a few elementary observations. First, note that 
\begin{equation}\label{elem1}
\left|\sqrt[n]{1+x}-1\right| \le 2\left|x\right|
\end{equation} 
whenever the left--hand side is well-defined. Conversely, a Taylor expansion argument shows that whenever $\eta>0$ is small enough, 
\begin{equation}\label{elem2}
\left|\sqrt[n]{1+x}-1\right|<  \eta \qquad \textrm{implies that} \qquad x\in\left(-2n\eta, 2n\eta\right).
\end{equation} 
Also, given $\bm{x}, \bm{y}\in\R^n$ expressed in polar coordinates as $\bm{x}=r\bm{u}$ and $ \bm{y}=\rho\bm{v}$ with $r,\rho>0$ and $\bm{u}, \bm{v}\in\Sph^d$, it holds that\footnote{To see this, in view of~\eqref{elem0},  it is enough to note the following~: when working in the plane spanned by the vectors $\bm{u}$ and $\bm{v}$, one obtains from the usual formula for the distance between two planar points expressed in polar coordinates that  $$\left\|\bm{x}-\bm{y}\right\|_2\;\asymp\; \left|\rho-r\right|+\sqrt{\rho r}\left\|\bm{u}-\bm{v}\right\|_2.$$} 
\begin{equation}\label{elem3}
\left\|\bm{x}-\bm{y}\right\|_2\;\asymp \; \left|r-\rho\right|+\sqrt{r\rho}\cdot\dsph(\bm{u},\bm{v}).
\end{equation} 

\begin{proof}[Proof of Proposition~\ref{cnsreldens}]
Assume first that~\eqref{reldens} holds and fix $\bm{x}\in\R^n$. The goal is to show that there exists an index $l\ge 1$ such that 
\begin{equation}\label{inegeprouver}
\left\|\bm{x}-\bm{s}_l\right\|_2\ll 1
\end{equation} 
(with an absolute implicit constant). Since the sequence $\left(\sqrt[n]{k+1}-\sqrt[n]{k}\right)_{k\ge 1}$ tends to zero as $k$ tends to infinity, it is enough to establish this fact when $\bm{x}$ is of the shape $\bm{x}=\sqrt[n]{k}\bm{v}$ for some $k\ge 1$ and some $\bm{v}\in\Sph^d$. 

Let $l\ge 1$. Then, 
\begin{align}
\left\|\bm{x}-\bm{s}_l\right\|_2\; &=\; \left\| \sqrt[n]{k}\bm{v}+\sqrt[n]{k}\bm{u}_l-\sqrt[n]{k}\bm{u}_l-\sqrt[n]{l}\bm{u}_l \right\|_2\nonumber\\
&\le\; \left|\sqrt[n]{k}-\sqrt[n]{l}\right|+\sqrt[n]{k}\left\|\bm{v}-\bm{u}_l\right\|_2\nonumber\\
&\underset{\eqref{elem0}}{\ll}\left|\sqrt[n]{k}-\sqrt[n]{l}\right|+\sqrt[n]{k}\cdot\dsph\left(\bm{v},\bm{u}_l\right).\label{ineqreldens}
\end{align}
Set $l=k+m$ for some $m\in\Z$. It then follows from~\eqref{elem1} that $$\left|\sqrt[n]{k}-\sqrt[n]{l}\right|\;=\;\sqrt[n]{k}\left|\sqrt[n]{1+\frac{m}{k}}-1\right|\;\le\; \frac{2\left|m\right|}{k^{1-1/n}}\cdotp$$ Upon choosing $m$ such that~\eqref{reldens} holds, this and~\eqref{ineqreldens} are easily seen to imply~\eqref{inegeprouver}.

Conversely, assume that the spiral set $\widetilde{\mathfrak{S}}$ defined by~\eqref{sk} is relatively dense. The goal is to show that for some absolute constants $c, C>0$, \eqref{reldens} holds for all $k\ge 1$ and all $\bm{v}\in\Sph^d$. To this end, fix  $k\ge 1$ and $\bm{v}\in\Sph^d$. By assumption, there exists an index $l\ge 1$ such that $$\left\|\bm{s}_l- \sqrt[n]{k}\bm{v}\right\|_2\,\ll\,1,$$ where the implicit constant is absolute. Decomposing the integer $l$ as $l=k+m$ for some $m\in\Z$, one thus obtains that $$\left\|\sqrt[n]{1+\frac{m}{k}}\bm{u}_{k+m}- \bm{v}\right\|_2\,\ll\,\frac{1}{\sqrt[n]{k}}\cdotp$$ From~\eqref{elem3}, this implies that $$\left|\sqrt[n]{1+\frac{m}{k}}-1\right|\,\ll\,\frac{1}{\sqrt[n]{k}}\qquad\textrm{and}\qquad \dsph\left(\bm{u}_{k+m}, \bm{v}\right)\,\ll\,\frac{1}{\sqrt[n]{k}}\cdotp$$ When $k$ is large enough, it follows from~\eqref{elem2} and from  the first inequality above that $\left|m\right|\ll k^{1-1/n}$. This suffices to complete the proof.
\end{proof}

The next proposition characterizes the cases when the spiral set $\widetilde{\mathfrak{S}}$ is  uniformly discrete.

\begin{prop}[Necessary and sufficient condition for uniform discreteness] \label{cnsunifdiscr} The set $\widetilde{\mathfrak{S}}$ defined in~\eqref{sk} is uniformly discrete if, and only if, there exists $\kappa>0$ such that the following holds~:
\begin{equation}\label{unifdiscr}
\liminf_{k\rightarrow\infty}\;\min_{1\le |m|\le\kappa k^{1-1/n}}\; \sqrt[n]{k}\cdot\dsph\left(\bm{u}_{k+m}, \bm{u}_{k}\right)\;>\; 0.
\end{equation} 
\end{prop}

Condition~\eqref{unifdiscr} can be rephrased in terms of a packing property of the spherical sequence $\left(\bm{u}_{k}\right)_{k\ge 1}$~: it is saying that there exist constants $\gamma,\kappa>0$ and an integer $k_0\ge 1$ such that for all $k\ge k_0$, the spherical caps with radius $\gamma k^{-1/n}$ centered at the points $\bm{u}_{k+m}$, where $ |m|\le \kappa k^{1-1/n}$,  are pairwise disjoint. In other words, 
\begin{equation}\label{packsep}
\mathcal{C}_d\left(\bm{u}_{l}, \frac{\gamma}{\sqrt[n]{k}}\right)\cap \mathcal{C}_d\left(\bm{u}_{k}, \frac{\gamma}{\sqrt[n]{k}}\right)\; =\; \emptyset
\end{equation} 
for all $k\ge k_0$ and $l\in\Z$ such that $1\le \left|l-k\right|\le \kappa k^{1-1/n}$.

\begin{proof}[Proof of Proposition~\eqref{cnsunifdiscr}] To show that~\eqref{unifdiscr} implies the uniform discreteness of the set $\widetilde{\mathfrak{S}}$, we argue by contraposition, assuming that this set is not uniformly discrete. Then, given $\epsilon>0$, there exist integers $k,m\ge 1$ such that $$\left\|\sqrt[n]{k+m}\bm{u}_{k+m}-\sqrt[n]{k}\bm{u}_k\right\|_2\;<\;\epsilon;$$ that is, such that $$\left\|\sqrt[n]{1+\frac{m}{k}}\bm{u}_{k+m}-\bm{u}_k\right\|_2\;<\;\frac{\epsilon}{\sqrt[n]{k}}\cdotp$$ From~\eqref{elem3}, this implies on the one hand that 
\begin{equation}\label{lesespi}
\sqrt[n]{k}\cdot\dsph\left(\bm{u}_{m+k}, \bm{u}_k\right)\;\ll\; \epsilon
\end{equation} 
and, on the other, that $$\left|\sqrt[n]{1+\frac{m}{k}}-1\right|\;\ll\; \frac{\epsilon}{\sqrt[n]{k}}\cdotp$$ It then follows from~\eqref{elem2} that for $k$ larger than some integer $k_0$, it holds that $|m|\ll k^{1-1/n}$. As \eqref{lesespi} holds for any $\epsilon>0$ for some integers $k$ and $m$ satisfying these constraints, one obtains that $$\liminf_{k\rightarrow\infty}\;\min_{1\le |m|\ll k^{1-1/n}}\; \sqrt[n]{k}\cdot\dsph\left(\bm{u}_{k+m}, \bm{u}_{k}\right)\; =\; 0,$$ which concludes the reasoning by contraposition.

For the converse implication, assume that the set $\widetilde{\mathfrak{S}}$ is uniformly discrete. Then, there exists $r>0$ such that for all integers $m, k$ with $m\neq 0$ and $m+k\ge 0$, it holds that $$\left\|\sqrt[n]{k+m}\bm{u}_{k+m}-\sqrt[n]{k}\bm{u}_k\right\|_2\;>\; r.$$ Applying~\eqref{elem2} once again, one can guarantee the existence of a constant $\alpha>0$ such that 
\begin{equation}\label{biggerr}
\left|\sqrt[n]{1+\frac{m}{k}}-1\right|+\sqrt[2n]{1+\frac{m}{k}}\cdot \dsph\left(\bm{u}_{m+k}, \bm{u}_k\right)\; >\;\frac{r\alpha}{\sqrt[n]{k}}\cdotp
\end{equation} 
If we now assume that $1\le |m|\le (r\alpha k^{1-1/n})/4$, then $$\left|\sqrt[n]{1+\frac{m}{k}}-1\right|\;\underset{\eqref{elem1}}{\le}\; \frac{r\alpha}{2\sqrt[n]{k}},$$ and~\eqref{biggerr} implies that $$\frac{r\alpha}{2\sqrt[n]{k}}\;<\; \sqrt[2n]{1+\frac{m}{k}}\cdot \dsph\left(\bm{u}_{m+k}, \bm{u}_k\right)\;\ll\; \dsph\left(\bm{u}_{m+k}, \bm{u}_k\right).$$ This shows that~\eqref{unifdiscr} holds with $\kappa=r\alpha/4$.
\end{proof}

\section{Lifting a Toral Sequence to the Sphere}\label{liftsection}

In practice, the density condition induced by Proposition~\ref{cnsreldens} and the separation condition induced by Proposition~\ref{cnsunifdiscr} are more naturally realised when the underlying space is the torus $\T^d$. Upon identifying the unit square $\K_d$ with $\T^d$ in the natural way, the following proposition provides a way to lift a well--seperated and well--dispersed toral sequence  to the sphere so that the lifted sequence still enjoys the same properties. 

\begin{prop}[Lifting a toral sequence to the sphere]\label{lifting}
Assume that $\tau_1, \tau_2~: \overline{\K}_d \longrightarrow \Sph^d$ are two maps which are bi--Lipschitz  to their images and which satisfy these conditions~:
\begin{equation}\label{condphi}
\tau_1\left(\K_d^{\circ}\right)\cap\tau_2\left(\K_d^{\circ}\right)=\emptyset, \qquad  \tau_1\left(\overline{\K}_d\right)\cup\tau_2\left(\overline{\K}_d\right)=\Sph^{d}\qquad \textrm{and}\qquad \tau_1\vert_{\partial\K_d}=\tau_2\vert_{\partial\K_d}
\end{equation}
(here, $\tau_i\vert_{\partial\K_d}$ denotes the restriction of the map $\tau_i$ to $\partial\K_d$).
Let $\left(\bm{x}_k\right)_{k\ge 1}$ be a sequence in $\K_d$ such that for some $\kappa>0$, the following separation condition holds~: 
\begin{equation}\label{unifdiscrbis}
\liminf_{k\rightarrow\infty}\;\min_{1\le |m|\le\kappa k^{1-1/n}}\; \sqrt[n]{k}\cdot\left\|\bm{x}_{k+m}- \bm{x}_{k}\right\|_2 \;>\; 0.
\end{equation}
Assume furthermore  that $\left(\bm{x}_k\right)_{k\ge 1}$ satisfies a density condition; namely, that there exists a partition of the natural integers into two sets $\mathcal{N}_1$ and $\mathcal{N}_2$ such that the following holds for some constants $c,C>0$~: for any integer $k\ge 1$ and any vector $\bm{v}\in\Sph^d$ such that $\bm{v}=\tau_i(\bm{y})$ for some $i\in\left\{1,2\right\}$ and some $\bm{y}\in\overline{\K}_d$, there exists an integer $m_i$ such that 
\begin{equation}\label{conddensx_n}
k+m_i\in\mathcal{N}_i, \qquad  \left|m_i\right|\le ck^{1-1/n} \qquad \textrm{and}\qquad \left\|\bm{x}_{k+m_i}-\bm{y}\right\|_2\;\le\; \frac{C}{\sqrt[n]{k}}\cdotp
\end{equation}
Then, the spherical sequence defined for all $k\ge 1$ by 
\begin{equation}\label{defuk}
\bm{u}_k\;=\; \left\{
    \begin{array}{lll}
        \tau_1\left(\bm{x}_k\right) & \mbox{if } & k\in\mathcal{N}_1,\\
        \tau_2\left(\bm{x}_k\right) & \mbox{if } & k\in\mathcal{N}_2
    \end{array}
\right.
\end{equation}
satisfies the density condition~\eqref{reldens} and the separation condition~\eqref{unifdiscr} (with possibly different values for the various constants).
\end{prop}

It should be noted that any circle sequence $\left(\bm{u}_k\right)_{k\ge 1}=\left(e(y_k)\right)_{k\ge 1}$, where $\left(y_k\right)_{k\ge 1}\in\R^\N$, can always be realised as the lift of a one--dimensional toral sequence (in the sense of the above proposition). To see this, notice that $\left(\bm{u}_k\right)_{k\ge 1}$ is also the sequence defined by~\eqref{defuk} when setting 
\begin{equation*}
\bm{x}_k\;=\; \left\{
    \begin{array}{lll}
        2\left\{y_k\right\} & \mbox{if } & 0\le \left\{y_k\right\} < 1/2,\\
        2-2\left\{y_k\right\} & \mbox{if } &  1/2\le \left\{y_k\right\}  <1,
    \end{array}
\right.
\end{equation*}

$$\mathcal{N}_1=\left\{k\ge 1 : 0\le \left\{y_k\right\}< \frac{1}{2}\right\}, \; \qquad \; \mathcal{N}_2=\left\{k\ge 1 :  \frac{1}{2}\le \left\{y_k\right\} <1 \right\}$$ and $$\tau_1(x)=\tau_2(-x)=e(x/2).$$ 

The effect of this lift is more apparent in the case that the sequence $\left(y_k\right)_{k\ge 1}$ arises from a linear toral flow such as when considering the sunflower spiral (cf.~Figure~\ref{sunflower}). Then, $\left(y_k\right)_{k\ge 1}=\left(\varphi k\right)_{k\ge 1}$ (recall that $\varphi$ stands for the golden ratio) and the sequence $\left(x_k\right)_{k\ge 1}$ above is just the discrete one--dimensional billiard $\left(2\left\|\varphi k\right\|\right)_{k\ge 1}$ in the unit interval $[0,1)$. The lift  then  maps this discrete billiard onto the northern or the southern hemisphere of the circle depending on the direction of the continuous flow $\left\{2\left\|\varphi t\right\|\right\}_{t\ge 0}$ in $[0,1)$. This is represented in Figure~\ref{Sunflowerbis?} where, for the sake of illustration and without loss of generality, the interval  $[0,1)$ is mapped onto the interval $[-1,1)$ so as to become the diameter  of the unit circle.

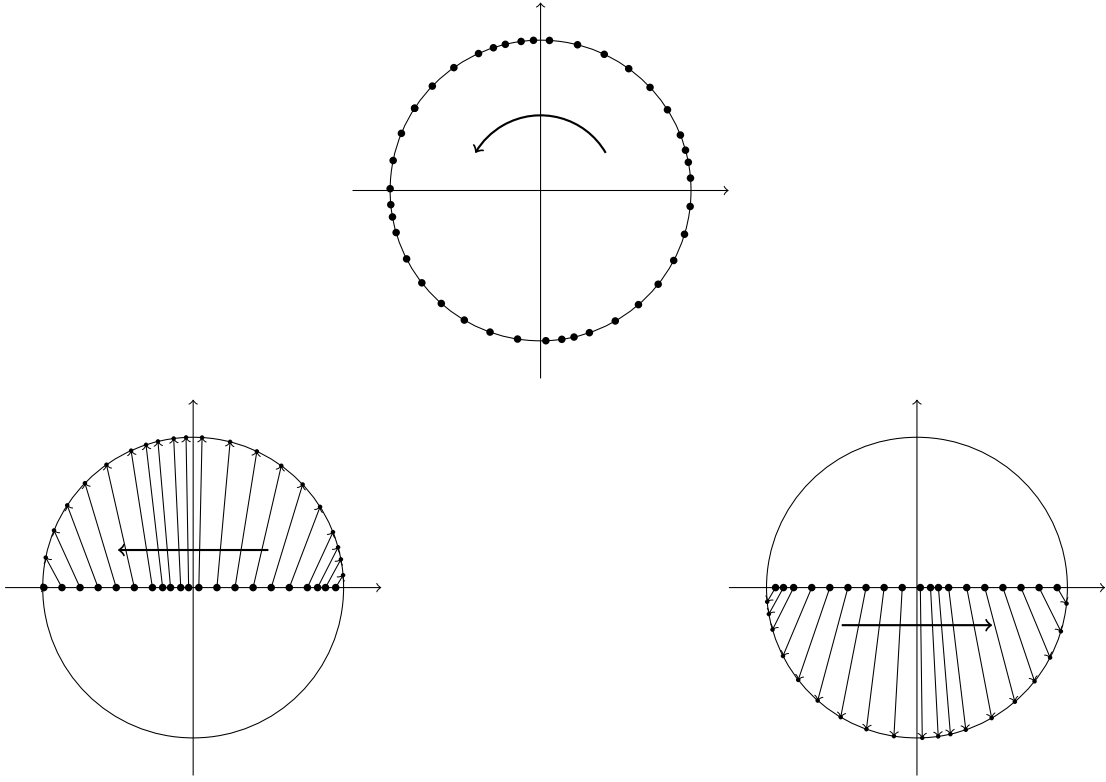
\begin{figure}[h!]

\begin{subfigure}{0\textwidth}
  \centering
\begin{tikzpicture}
\draw[] (0,0) circle (2) ;

\fill (2*-0.0472200962543599, 2*0.998884509094885) circle (0.05) ;
\fill (2*-0.995540525019458, 2*-0.094334845332899) circle (0.05) ;
\fill (2*0.141239135087429, 2*-0.989975508141365) circle (0.05) ;
\fill (2*0.982201873912036, 2*0.187828322900687) circle (0.05) ;
\fill (2*-0.233998469142107, 2*0.972236965168035) circle (0.05) ;
\fill (2*-0.96010301343951, 2*-0.27964656905525) circle (0.05) ;
\fill (2*0.324670782559537, 2*-0.945827089352054) circle (0.05) ;
\fill (2*0.92944104223263, 2*0.36897066145362) circle (0.05) ;
\fill (2*-0.412447373513491, 2*0.910981429054304) circle (0.05) ;
\fill (2*-0.8904894328783, 2*-0.455003922985377) circle (0.05) ;
\fill (2*0.4965453669815, 2*-0.868010770975342) circle (0.05) ;
\fill (2*0.843595592831255, 2*0.536979027295931) circle (0.05) ;
\fill (2*-0.576214697167989, 2*0.817298368264371) circle (0.05) ;
\fill (2*-0.789177765904356, 2*-0.61416484253188) circle (0.05) ;
\fill (2*0.6507447973036, 2*-0.75929652230357) circle (0.05) ;
\fill (2*0.727721301972956, 2*0.685872952269431) circle (0.05) ;
\fill (2*-0.71947093715462, 2*0.694522548654724) circle (0.05) ;
\fill (2*-0.659774328163645, 2*-0.75146379546603) circle (0.05) ;
\fill (2*0.781780151718708, 2*-0.623554163147576) circle (0.05) ;
\fill (2*0.585942860135834, 2*0.810352370673301) circle (0.05) ;
\fill (2*-0.837116708229044, 2*0.547024329261295) circle (0.05) ;
\fill (2*-0.837116708229044, 2*0.547024329261295) circle (0.05) ;
\fill (2*-0.506885397058414, 2*-0.862013453635692) circle (0.05) ;
\fill (2*0.884987062707102, 2*-0.465615612754831) circle (0.05) ;
\fill (2*0.423307048488632, 2*0.905986281739322) circle (0.05) ;
\fill (2*-0.924964261856665, 2*0.380054093897113) circle (0.05) ;
\fill (2*-0.335953245535206, 2*-0.941878663530691) circle (0.05) ;
\fill (2*0.956691751038939, 2*-0.291102891593416) circle (0.05) ;
\fill (2*0.245603092395581, 2*0.96937047665262) circle (0.05) ;
\fill (2*-0.979886554365514, 2*0.199555357166076) circle (0.05) ;
\fill (2*-0.153062417564592, 2*-0.988216522999531) circle (0.05) ;
\fill (2*0.994341798546165, 2*-0.106227998493699) circle (0.05) ;
\fill (2*0.0591565866904294, 2*0.998248715627091) circle (0.05) ;
\fill (2*-0.999928557981368, 2*0.011953197618298) circle (0.05) ;
\fill (2*0.035276858820294, 2*-0.999377577911258) circle (0.05) ;
\fill (2*0.996597004643276, 2*0.0824282132285429) circle (0.05) ;
\fill (2*-0.129395671792422, 2*0.991593041585805) circle (0.05) ;
\fill (2*-0.984376852489405, 2*-0.17607445096621) circle (0.05) ;
\fill (2*0.222360411242653, 2*-0.974964536540688) circle (0.05) ;
\fill (2*0.96337709244533, 2*0.268150289486291) circle (0.05) ;
\fill (2*-0.313341929311677, 2*0.949640371580335) circle (0.05) ;

\draw[thin,->] (0,-2.5) -- (0,2.5);
\draw[thin,->] (-2.5,0) -- (2.5,0);

\draw[thick, ->] (0.86602540378,0.5) arc (30:150:1);
\end{tikzpicture}
\end{subfigure}
\newline
\vspace{2mm}

\begin{subfigure}{.33\textwidth}
	\centering
\begin{tikzpicture}
\draw (0,0) circle (2) ;

\fill (2*-0.0472200962543599, 2*0.998884509094885) circle (0.03) ;
\draw[->] (2*-0.030072431,0) -- (2*-0.0472200962543599, 2*0.998884509094885);
\fill (2*-0.030072431,0) circle (0.05) ;

\fill (2*0.982201873912036, 2*0.187828322900687) circle (0.03) ;
\draw[->] (2*0.879710281,0) -- (2*0.982201873912036, 2*0.187828322900687);
\fill (2*0.879710281,0) circle (0.05) ;

\fill (2*-0.233998469142107, 2*0.972236965168035) circle (0.03) ;
\draw[->] (2*-0.150362149,0) -- (2*-0.233998469142107, 2*0.972236965168035);
\fill (2*-0.150362149,0) circle (0.05) ;

\fill (2*0.92944104223263, 2*0.36897066145362) circle (0.03) ;
\draw[->] (2*0.759420563,0) -- (2*0.92944104223263, 2*0.36897066145362);
\fill (2*0.759420563,0) circle (0.05) ;

\fill (2*-0.412447373513491, 2*0.910981429054304) circle (0.03) ;
\draw[->] (2*-0.270651868,0) -- (2*-0.412447373513491, 2*0.910981429054304);
\fill (2*-0.270651868,0) circle (0.05) ;

\fill (2*0.843595592831255, 2*0.536979027295931) circle (0.03) ;
\draw[->] (2*0.639130844,0) -- (2*0.843595592831255, 2*0.536979027295931);
\fill (2*0.639130844,0) circle (0.05) ;

\fill (2*-0.576214697167989, 2*0.817298368264371) circle (0.03) ;
\draw[->] (2*-0.390941586,0) -- (2*-0.576214697167989, 2*0.817298368264371);
\fill (2*-0.390941586,0) circle (0.05) ;

\fill (2*0.727721301972956, 2*0.685872952269431) circle (0.03) ;
\draw[->] (2*0.518841126, 0) -- (2*0.727721301972956, 2*0.685872952269431);
\fill (2*0.518841126, 0)  circle (0.05) ;

\fill (2*-0.71947093715462, 2*0.694522548654724) circle (0.03) ;
\draw[->] (2*-0.511231305, 0) -- (2*-0.71947093715462, 2*0.694522548654724);
\fill (2*-0.511231305, 0) circle (0.05) ;

\fill (2*0.585942860135834, 2*0.810352370673301) circle (0.03) ;
\draw[->] (2*0.398551407, 0) -- (2*0.585942860135834, 2*0.810352370673301);
\fill (2*0.398551407, 0) circle (0.05) ;

\fill (2*-0.837116708229044, 2*0.547024329261295) circle (0.03) ;
\draw[->] (2*-0.631521023, 0) -- (2*-0.837116708229044, 2*0.547024329261295) ;
\fill (2*-0.631521023, 0) circle (0.05) ;

\fill (2*0.423307048488632, 2*0.905986281739322) circle (0.03) ;
\draw[->] (2*0.278261689, 0) -- (2*0.423307048488632, 2*0.905986281739322);
\fill (2*0.278261689, 0) circle (0.05) ;

\fill (2*-0.924964261856665, 2*0.380054093897113) circle (0.03) ;
\draw[->] (2*-0.751810742, 0) -- (2*-0.924964261856665, 2*0.380054093897113);
\fill (2*-0.751810742, 0) circle (0.05) ;

\fill (2*0.245603092395581, 2*0.96937047665262) circle (0.03) ;
\draw[->] (2*0.15797197, 0) -- (2*0.245603092395581, 2*0.96937047665262) ;
\fill (2*0.15797197, 0)  circle (0.05) ;

\fill (2*-0.979886554365514, 2*0.199555357166076) circle (0.03) ;
\draw[->] (2*-0.872100461, 0) -- (2*-0.979886554365514, 2*0.199555357166076);
\fill (2*-0.872100461, 0) circle (0.05) ;

\fill (2*0.0591565866904294, 2*0.998248715627091) circle (0.03) ;
\draw[->] (2*0.037682252, 0) -- (2*0.0591565866904294, 2*0.998248715627091) ;
\fill (2*0.037682252, 0) circle (0.05) ;

\fill (2*-0.999928557981368, 2*0.011953197618298) circle (0.03) ;
\draw[->] (2*-0.992390179, 0) -- (2*-0.999928557981368, 2*0.011953197618298);
\fill (2*-0.992390179, 0) circle (0.05) ;

\fill (2*0.996597004643276, 2*0.0824282132285429) circle (0.03) ;
\draw[->] (2*0.947464964, 0) -- (2*0.996597004643276, 2*0.0824282132285429) ;
\fill (2*0.947464964, 0) circle (0.05) ;

\fill (2*-0.129395671792422, 2*0.991593041585805) circle (0.03) ;
\draw[->] (2*-0.082607467, 0) -- (2*-0.129395671792422, 2*0.991593041585805);
\fill (2*-0.082607467, 0) circle (0.05) ;

\fill (2*0.96337709244533, 2*0.268150289486291) circle (0.03) ;
\draw[->] (2*0.827175245, 0) -- (2*0.96337709244533, 2*0.268150289486291);
\fill (2*0.827175245, 0) circle (0.05) ;

\fill (2*-0.313341929311677, 2*0.949640371580335) circle (0.03) ;
\draw[->] (2*-0.202897185, 0) -- (2*-0.313341929311677, 2*0.949640371580335);
\fill (2*-0.202897185, 0) circle (0.05) ;

\draw[thin,->] (0,-2.5) -- (0,2.5);
\draw[thin,->] (-2.5,0) -- (2.5,0);
\draw[thick,->] (1,0.5) -- (-1,0.5);

\end{tikzpicture}
\end{subfigure}
\hspace{40mm}
\begin{subfigure}{.33\textwidth}
	\centering
\begin{tikzpicture}
\draw (0,0) circle (2) ;

\fill (2*-0.995540525019458, 2*-0.094334845332899) circle (0.03) ;
\draw[->] (2*-0.939855143,0) -- (2*-0.995540525019458, 2*-0.094334845332899);
\fill (2*-0.939855143,0)  circle (0.05) ;

\fill (2*0.141239135087429, 2*-0.989975508141365) circle (0.03) ;
\draw[->] (2*0.090217288,0) --  (2*0.141239135087429, 2*-0.989975508141365);
\fill (2*0.090217288,0) circle (0.05) ;

\fill (2*-0.96010301343951, 2*-0.27964656905525) circle (0.03) ;
\draw[->] (2*-0.819565424, 0) -- (2*-0.96010301343951, 2*-0.27964656905525);
\fill (2*-0.819565424, 0) circle (0.05) ;

\fill (2*0.324670782559537, 2*-0.945827089352054) circle (0.03) ;
\draw[->] (2*0.210507006,0) -- (2*0.324670782559537, 2*-0.945827089352054);
\fill (2*0.210507006,0) circle (0.05) ;

\fill (2*-0.8904894328783, 2*-0.455003922985377) circle (0.03) ;
\draw[->] (2*-0.699275706,0) -- (2*-0.8904894328783, 2*-0.455003922985377);
\fill (2*-0.699275706,0) circle (0.05) ;

\fill (2*0.4965453669815, 2*-0.868010770975342) circle (0.03) ;
\draw[->] (2*0.330796725,0) --  (2*0.4965453669815, 2*-0.868010770975342);
\fill (2*0.330796725,0) circle (0.05) ;

\fill (2*-0.789177765904356, 2*-0.61416484253188) circle (0.03) ;
\draw[->] (2*-0.578985987,0) -- (2*-0.789177765904356, 2*-0.61416484253188);
\fill (2*-0.578985987,0)  circle (0.05) ;

\fill (2*0.6507447973036, 2*-0.75929652230357) circle (0.03) ;
\draw[->] (2*0.451086443,0) -- ((2*0.6507447973036, 2*-0.75929652230357);
\fill (2*0.451086443,0) circle (0.05) ;

\fill (2*-0.659774328163645, 2*-0.75146379546603) circle (0.03) ;
\draw[->] (2*-0.458696269, 0) -- (2*-0.659774328163645, 2*-0.75146379546603);
\fill (2*-0.458696269, 0) circle (0.05) ;

\fill (2*0.781780151718708, 2*-0.623554163147576) circle (0.03) ;
\draw[->] (2*0.571376162,0) -- (2*0.781780151718708, 2*-0.623554163147576);
\fill (2*0.571376162,0) circle (0.05) ;

\fill (2*-0.506885397058414, 2*-0.862013453635692) circle (0.03) ;
\draw[->] (2*-0.33840655,0) -- (2*-0.506885397058414, 2*-0.862013453635692);
\fill (2*-0.33840655,0) circle (0.05) ;

\fill (2*0.884987062707102, 2*-0.465615612754831) circle (0.03) ;
\draw[->] (2*0.69166588,0) --  (2*0.884987062707102, 2*-0.465615612754831);
\fill (2*0.69166588,0) circle (0.05) ;

\fill (2*-0.335953245535206, 2*-0.941878663530691) circle (0.03) ;
\draw[->] (2*-0.218116832,0) -- (2*-0.335953245535206, 2*-0.941878663530691);
\fill (2*-0.218116832,0) circle (0.05) ;

\fill (2*0.956691751038939, 2*-0.291102891593416) circle (0.03) ;
\draw[->] (2*0.811955599,0) -- (2*0.956691751038939, 2*-0.291102891593416);
\fill (2*0.811955599,0) circle (0.05) ;

\fill (2*-0.153062417564592, 2*-0.988216522999531) circle (0.03) ;
\draw[->] (2*-0.097827113,0) -- (2*-0.153062417564592, 2*-0.988216522999531);
\fill (2*-0.097827113,0) circle (0.05) ;

\fill (2*0.994341798546165, 2*-0.106227998493699) circle (0.03) ;
\draw[->] (2*0.932245318,0) -- (2*0.994341798546165, 2*-0.106227998493699);
\fill (2*0.932245318,0) circle (0.05) ;

\fill (2*0.035276858820294, 2*-0.999377577911258) circle (0.03) ;
\draw[->] (2*0.022462605,0) -- (2*0.035276858820294, 2*-0.999377577911258);
\fill (2*0.022462605,0) circle (0.05) ;

\fill (2*-0.984376852489405, 2*-0.17607445096621) circle (0.03) ;
\draw[->] (2*-0.887320107,0) -- (2*-0.984376852489405, 2*-0.17607445096621);
\fill (2*-0.887320107,0) circle (0.05) ;

\fill (2*0.222360411242653, 2*-0.974964536540688) circle (0.03) ;
\draw[->] (2*0.142752324,0) -- (2*0.222360411242653, 2*-0.974964536540688);
\fill (2*0.142752324,0) circle (0.05) ;

\draw[thin,->] (0,-2.5) -- (0,2.5);
\draw[thin,->] (-2.5,0) -- (2.5,0);
\draw[thick,->] (-1,-0.5) -- (1,-0.5);

\end{tikzpicture}
\end{subfigure}
\caption{The subsequence $\left\{e(n\varphi)\right\}_{1\le n\le 40}$ of the circle rotation obtained from  the Golden Ratio  $\varphi=(1+\sqrt{5})/2$ is represented at the top. It can be obtained as the lift of the discrete one--dimensional billiard  $\left(-4\left\|\varphi k\right\|+1\right)_{k\ge 1}$  in the interval $[-1, 1)$ (bottom pictures)~: each time the billiard flow $\left\{-4\left\|\varphi t\right\|+1\right\}_{t\ge 0}$  ``hits'' an end of the interval, the hemisphere onto which the lifting operates is changed.\\}
\label{Sunflowerbis?}
\end{figure}

\begin{proof}[Proof of Proposition~\ref{lifting}] For the sake of simplicity of notation, set $$\tau\;=\;\tau_1\vert_{\partial\K_d}\;=\;\tau_2\vert_{\partial\K_d}.$$ The goal is to show that the sequence $\left(\bm{u}_k\right)_{k\ge 1}$ defined in~\eqref{defuk} satisfies both conditions~\eqref{reldens} and~\eqref{unifdiscr}. These will be established separately.

\begin{proof}[Proof that the density condition~\eqref{reldens} holds.] Let $k\ge 1$ and $\bm{v}\in\Sph^d$. Assume that $\bm{v}=\tau_i\left(\bm{y}\right)$ for some $\bm{y}\in\overline{\K}_d$ and some $i\in\{1,2\}$. Then, by assumption, there exists an integer $m_i$ such that $\left|m_i\right|\le ck^{1-1/n}$, $k+m_i\in\mathcal{N}_i$ and $\left\|\bm{x}_{k+m_i}-\bm{y}\right\|_2\le Ck^{-1/n}$. Since $\tau_i$ is Lipschitz, this implies that $$\dsph\left(\bm{u}_{k+m_i},\bm{v}\right)\;=\; \dsph\left(\tau_i\left(\bm{x}_{k+m_i}\right), \tau_i\left(\bm{y}\right)\right)\;\ll\; \left\|\bm{x}_{k+m_i} - \bm{y}\right\|_2 \;\le\;\frac{C}{\sqrt[n]{k}},$$ thereby establishing the required density conditon.

\end{proof}

\begin{proof}[Proof that the separation condition~\eqref{unifdiscr} holds.] This part is  more involved. Let $k$ be a large enough integer and let $m$ be an integer such that $1\le \left|m\right|\le \kappa k^{1-1/n}$. Let $i,j\in\left\{1,2\right\}$ be indices such that $\bm{u}_k=\tau_i\left(\bm{x}_k\right)$ and $\bm{u}_{k+m}=\tau_j\left(\bm{x}_{k+m}\right)$. 

The case when $i=j$ can be easily dealt with~: since $\tau_i$ is bi--Lipschitz, one obtains that $$\dsph\left(\bm{u}_{k+m}, \bm{u}_k\right)\;=\; \dsph\left(\tau_i\left(\bm{x}_{k+m}\right),\tau_i\left(\bm{x}_{k}\right)\right) \;\gg\; \left\|\bm{x}_{k+m}-\bm{x}_{k}\right\|_2 \;\underset{\eqref{unifdiscrbis}}{\gg}\; \frac{1}{\sqrt[n]{k}}, $$ whence the conclusion.

Assume now that $i\neq j$. Since $\tau_i$ and $\tau_j$ coincide on the boundary of $\K_d$, in view of the previous case, one can furthermore assume without loss of generality that $\bm{x}_k\in\K_d^\circ$ and $\bm{x}_{k+m}\in\K_d^\circ$.  Then, $\bm{u}_k\in\tau_i\left(\K_d^\circ\right)$ and $\bm{u}_{k+m}\in\tau_j\left(\K_d^\circ\right)$. Consider the geodesic arc $t\in[0,1]\mapsto \bm{v}(t)$ joining $\bm{v}(0)=\bm{u}_k$ to $\bm{v}(1)=\bm{u}_{k+m}$ and note that assumption~\eqref{condphi} implies that the decomposition $$\Sph^d\;=\; \tau_1\left(\K_d^\circ\right)\cup \tau_2\left(\K_d^\circ\right)\cup \tau\left(\partial\K_d\right)$$ holds, where the unions are pairwise disjoint. The Mean Value Theorem applied to the continuous map $t\in[0,1]\mapsto\textrm{dist}\left(\bm{v}(t), \tau_1\left(\K_d^\circ\right)\right)- \textrm{dist}\left(\bm{v}(t), \tau_2\left(\K_d^\circ\right)\right)$ then implies the existence of a value $t_0\in (0,1)$ and of a vector $\bm{z}\in\partial\K_d$ such that $\bm{v}\left(t_0\right)=\tau(\bm{z})$. 

Therefore, given that the maps $\tau_i $ and $\tau_j$ are bi--Lipschitz, one gets
\begin{align*}
\dsph\left(\bm{u}_{k+m}, \bm{u}_k\right)\; & =\; \dsph\left(\tau_j\left(\bm{x}_{k+m}\right),\tau_i\left(\bm{x}_{k}\right)\right)\; =\;  \dsph\left(\tau_j\left(\bm{x}_{k+m}\right),\tau\left(\bm{z}\right)\right)+\dsph\left(\tau_j\left(\bm{x}_{k}\right),\tau\left(\bm{z}\right)\right)\\
&\gg \;  \left\|\bm{x}_{k+m}-\bm{z}\right\|_2 + \left\|\bm{x}_{k}-\bm{z}\right\|_2 \;\ge\;  \left\|\bm{x}_{k+m}-\bm{x}_{k}\right\|_2 \\
&\underset{\eqref{unifdiscrbis}}{\gg}\; \frac{1}{\sqrt[n]{k}},
\end{align*}
whence the conclusion in this case also.
\end{proof}

This completes the proof of Proposition~\ref{lifting}.
\end{proof}

We conclude this section by noticing that one can find in any dimension $d\ge 1$ bi--Lipschitz maps $\tau_1$ and $\tau_2$ satisfying  assumptions~\eqref{condphi} in Proposition~\ref{lifting}. 

Indeed, denote by $\mathfrak{s}$ the contracting map characterized by its action in each direction in the following way~: $\mathfrak{s}$  maps a vector on the boundary of the hypercube $\K'_d:=[-1, 1)^d$ to the vector in the same direction normalised so as to lie on the unit sphere $\Sph^{d-1}$. Explicitly, if $\bm{x}=\rho\bm{u}$ with $\rho\ge 0$ and $\bm{u}\in\Sph^{d-1}$, then $$\mathfrak{s}(\bm{x})\:=\: \rho\left\|\bm{u}\right\|_{\infty}\bm{u}.$$ Working in polar coordinates with relations~\eqref{elem0} and~\eqref{elem3}, it is easily seen that $\mathfrak{s}$ is bi--Lipschitz between $\K'_d$ and its image (the unit ball in $\R^{d-1}$).

Let also $N=\left(\bm{0},1\right)\in\R^n$  and $S=\left(\bm{0}, -1\right)\in\R^n$ stand  respectively for the ``North Pole'' and the ``South Pole'' in $\Sph^d$. The stereographic projections with respect to $N$ and $S$ are denoted by $\mathfrak{p}_N$ and $\mathfrak{p}_S$, respectively. Explicitly, 
\begin{align*}
\mathfrak{p}_N\; :\; \bm{x}\in\R^{d}\;\mapsto \; \mathfrak{p}_N\left(\bm{x}\right)=\left(\frac{2\bm{x}}{1+\left\|\bm{x}\right\|_2^2}, \; \frac{\left\|\bm{x}\right\|_2^2-1}{1+\left\|\bm{x}\right\|_2^2}\right)\in \Sph^{d}\backslash \{ N \}
\end{align*}
and
\begin{align*}
\mathfrak{p}_S\; :\; \bm{x}\in\R^{d}\;\mapsto\; \mathfrak{p}_S\left(\bm{x}\right)\;=\; \left(\frac{2\bm{x}}{1+\left\|\bm{x}\right\|_2^2}, \; \frac{1-\left\|\bm{x}\right\|_2^2}{1+\left\|\bm{x}\right\|_2^2}\right) \in \Sph^{d}\backslash\{S\}.
\end{align*}

Finally, given $\bm{x}\in\K_d$, set $$\tau_1(\bm{x})=\left(\mathfrak{p}_N\circ\mathfrak{s}\right)\left(2\bm{x}-\bm{1}\right)\qquad\textrm{ and }\qquad \tau_2\left(\bm{x}\right)=\left(\mathfrak{p}_S\circ\mathfrak{s}\right)\left(2\bm{x}-\bm{1}\right)$$ (here, the effect of the transformation $\bm{x}\mapsto 2\bm{x}-\bm{1}$ is to map $\K_d$ onto $\K'_d$). Since a stereographic projection is bi--Lipschitz between a compact set and its image, $\tau_1$ and $\tau_2$ are bi--Lipschitz. Furthermore, $\tau_1$ and $\tau_2$  map the interior of $\K_d$ to the southern and the northern hemispheres of $\Sph^d$, respectively. In other words, with obvious notation,  $$\tau_1\left(\K_d^{\circ}\right)=\Sph^d\cap\left\{x_n<0\right\} \qquad\textrm{ and }\qquad \tau_2\left(\K_d^{\circ}\right)=\Sph^d\cap\left\{x_n>0\right\}.$$ It also holds that  $\tau_1$ and $\tau_2$ coincide when restricted to $\partial\K_d$, and the image of this boundary by either of these maps is the equator $\Sph^d\cap\left\{x_n=0\right\}$. This shows that all conditions in~\eqref{condphi} are indeed satisfied.

The lifting induced by the maps $\tau_1$ and $\tau_2$ above is described in Figure~\ref{projstretchedbilliard} when $d=2$ (i.e.~in $\R^3$).  We have made no attempt to minimize the distorsion caused by such lifting when defining our maps $\tau_1$ and $\tau_2$ (and it is indeed not hard to see that one can easily do better in this respect with the same ideas).

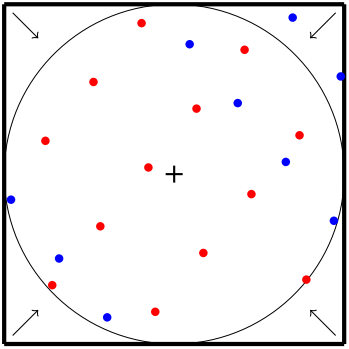
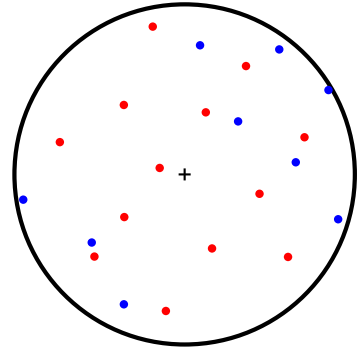
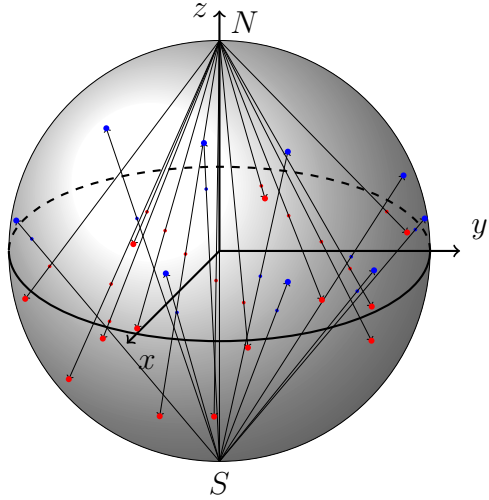
\begin{figure}[h!]
\begin{subfigure}{0.33\textwidth}
\centering

\begin{tikzpicture}[scale=1.13]
\draw[ultra thin] (2,2) circle (2);

\draw[ultra thick] (0,0) -- (0,4);
\draw[ultra thick] (4,0) -- (4,4);
\draw[ultra thick] (0,4) -- (4,4);
\draw[ultra thick] (0,0) -- (4,0);

\fill[color=red] (4*0.14142135623,4*0.17320508075) circle (0.05) ;
\fill[color=red] (8*0.14142135623,8*0.17320508075) circle (0.05) ;
\fill[color=red] (12*0.14142135623,12*0.17320508075) circle (0.05) ;
\fill[color=red] (16*0.14142135623,16*0.17320508075) circle (0.05) ;
\fill[color=red] (20*0.14142135623,20*0.17320508075) circle (0.05) ;
\fill[color=blue]  (24*0.14142135623,8-24*0.1732050807) circle (0.05) ;
\fill[color=blue]  (28*0.14142135623,8-28*0.1732050807) circle (0.05) ;
\fill[color=red] (8-32*0.14142135623,8-32*0.1732050807) circle (0.05) ;
\fill[color=red] (8-36*0.14142135623,8-36*0.1732050807) circle (0.05) ;
\fill[color=red] (8-40*0.14142135623,8-40*0.1732050807) circle (0.05) ;
\fill[color=red] (8-44*0.14142135623,8-44*0.1732050807) circle (0.05) ;
\fill[color=blue] (8-48*0.14142135623,-8+48*0.1732050807) circle (0.05) ;
\fill[color=blue]  (8-52*0.14142135623,-8+52*0.1732050807) circle (0.05) ;
\fill[color=blue]  (8-56*0.14142135623,-8+56*0.1732050807) circle (0.05) ;
\fill[color=red] (-8+60*0.14142135623,-8+60*0.1732050807) circle (0.05) ;
\fill[color=red] (-8+64*0.14142135623,-8+64*0.1732050807) circle (0.05) ;
\fill[color=red] (-8+68*0.14142135623,-8+68*0.1732050807) circle (0.05) ;
\fill[color=blue]  (-8+72*0.14142135623,16-72*0.1732050807) circle (0.05) ;
\fill[color=blue]  (-8+76*0.14142135623,16-76*0.1732050807) circle (0.05) ;
\fill[color=blue]  (-8+80*0.14142135623,16-80*0.1732050807) circle (0.05) ;
\fill[color=blue]  (-8+84*0.14142135623,16-84*0.1732050807) circle (0.05) ;
\fill[color=red] (16-88*0.14142135623,16-88*0.1732050807) circle (0.05) ;

\draw[thick]  (2,1.9) -- (2,2.1) ;
\draw[thick]  (1.9,2) -- (2.1,2) ;

\draw[->]  (0.1,0.1) -- (0.4,0.4) ;
\draw[->]  (3.9,3.9) -- (3.6,3.6) ;
\draw[->]  (3.9,0.1) -- (3.6,0.4) ;
\draw[->]  (0.1,3.9) -- (0.4,3.6) ;

\end{tikzpicture}
\caption{Two sets of points $\mathcal{N}_1$ (blue) and $\mathcal{N}_2$ (red) lie inside the square $\K'_2=[-1, 1)^2$.}
\label{billiard}
\end{subfigure}%
\hspace{50mm}
\begin{subfigure}{0.33\textwidth}
\centering
\vspace{8ex}

\begin{tikzpicture}[scale=0.8]
\draw[ultra thick] (2,2) circle (2.82842712475);

\fill[color=red] (0.500781305819615, 0.633669180764079) circle (0.07) ;
\fill[color=red] (0.997072811544766, 1.29065505170429) circle (0.07) ;
\fill[color=red] (1.58525727921859, 2.10741637157999) circle (0.07) ;
\fill[color=red] (2.35172467643424, 3.03249184731193) circle (0.07) ;
\fill[color=red] (3.0196617569812, 3.80207574187137) circle (0.07) ;

\fill[color=blue] (3.57241128420321, 4.0787968268709) circle (0.07) ;
\fill[color=blue] (4.39027931028539, 3.40291875034153) circle (0.07) ;

\fill[color=red] (3.9916428980629, 2.61786485392965) circle (0.07) ;
\fill[color=red] (3.24422816294235, 1.6777508875694) circle (0.07) ;
\fill[color=red] (2.45517305995988, 0.768764864439076) circle (0.07) ;
\fill[color=red] (1.68820581202952, -0.271171303280383) circle (0.07) ;

\fill[color=blue] (0.990171338200502, -0.160206249181023) circle (0.07) ;
\fill[color=blue] (0.456215898782451, 0.867357928290139) circle (0.07) ;
\fill[color=blue] (-0.682051246039307, 1.58012105609829) circle (0.07) ;

\fill[color=red] (-0.0737136049039098, 2.537081854051) circle (0.07) ;
\fill[color=red] (0.989731803244405, 3.15514129749943) circle (0.07) ;
\fill[color=red] (1.47004299279698, 4.45791087801629) circle (0.07) ;

\fill[color=blue] (2.25605067650332, 4.14745255229198) circle (0.07) ;
\fill[color=blue] (2.88852688283596, 2.88170384381359) circle (0.07) ;
\fill[color=blue] (3.84686455826186, 2.20186961379317) circle (0.07) ;
\fill[color=blue] (4.55115894902215, 1.2544592198622) circle (0.07) ;

\fill[color=red] (3.71814054874003, 0.627575313062749) circle (0.07) ;

\draw[thick]  (2,1.9) -- (2,2.1) ;
\draw[thick]  (1.9,2) -- (2.1,2) ;

\end{tikzpicture}
\caption{The square $\K'_2$ and the points inside it are radially contracted to the unit disk in $\R^{2}$.}\vspace{8ex}
\label{stretched_billiard}
\end{subfigure}%
\newline
\begin{subfigure}{0.8\textwidth}
\centering
\begin{tikzpicture}[scale=0.8]

\draw [ball color=white,very thin] (0,0,0) circle (3.5) ;

\draw[->, thick] (0,0,0)--(0,4,0) node[anchor=east] {$z$};
\draw[->, thick] (0,0,0)--(4,0,0) node[anchor=south west] {$y$};
\draw[->, thick] (0,0,0)--(0,0,4) node[anchor=north west] {$x$};
\draw (0,0,0)--(0,3.8,0) node[anchor=west] {$N$};
\draw (0,0,0)--(0,-3.5,0) node[anchor=north] {$S$};

\draw[thick] (3.5,0,0) arc (0:-180:3.5cm and 1.5cm);
\draw[thick, dashed] (3.5,0,0) arc (0:180:3.5cm and 1.4cm);

\fill[color=red] (-1.85518848397935,0,-1.69074812833648) circle(0.03);
\fill[color=red] (-1.24105907798137,0,-0.877769590490974) circle(0.03);
\fill[color=red] (-0.513217933043836,0,0.13292097831997) circle(0.03);
\fill[color=red] (0.435237081677996,0,1.27764347681299) circle(0.03);
\fill[color=red] (1.26176705002991,0,2.22995496024415) circle(0.03);
\fill[color=red]  (0.563247925291419,0,-1.52357574877502) circle(0.03);
\fill[color=red] (2.46453234810117,0,0.764568749119369) circle(0.03);
\fill[color=red]  (1.53965379987164,0,-0.39876293210221) circle(0.03);
\fill[color=red] (-0.385825623132825,0,-2.81043110211057) circle(0.03);
\fill[color=red]  (-2.56608966645273,0,0.664604886837016) circle(0.03);
\fill[color=red] (-1.2501431122937,0,1.42941442820087) circle(0.03);
\fill[color=red]  (-0.655788338675098,0,3.04150953642975) circle(0.03);
\fill[color=red] (2.12609045781879,0,-1.69828890489316) circle(0.03);

\fill[color=blue](1.94575969326994,0,2.57237983271818) circle(0.03);
\fill[color=blue](2.95781974114179,0,1.73602338317638) circle(0.03);
\fill[color=blue](-1.24959921553466,0,-2.6731188532116) circle(0.03);
\fill[color=blue](-1.91033536164312,0,-1.40157305672832) circle(0.03);
\fill[color=blue] (-3.31886909122007,0,-0.519573684890392) circle(0.03);
\fill[color=blue](0.316846546943423,0,2.65733695848966) circle(0.03);
\fill[color=blue](1.099495922204,0,1.09105284217227) circle(0.03);
\fill[color=blue](2.28537829287818,0,0.249800902448851) circle(0.03);
\fill[color=blue] (2.9,0,-0.922559647246643) circle(0.03);

\draw[ultra thin,->] (0,3.5,0)--(-2.20518201841681,-0.660298067344373,-2.00971890590983);
\draw[ultra thin,->] (0,3.5,0)--(-1.87939723683508,-1.80022337020569,-1.32924997062176);
\draw[ultra thin,->] (0,3.5,0)--(-0.903072433754454,-2.6586965587787,0.233891420505342);
\draw[ultra thin,->] (0,3.5,0)--(0.682000512109353,-1.98437137566493,2.00202037501091);
\draw[ultra thin,->] (0,3.5,0)--(1.47873108832973,-0.60183385992793,2.61340136058392);
\draw[ultra thin,->] (0,3.5,0)--(0.834173399159874,-1.68352002015628,-2.25642439885711);
\draw[ultra thin,->] (0,3.5,0)--(2.87399747167305,-0.581500962487159,0.891596595833803);
\draw[ultra thin,->] (0,3.5,0)--(2.29705073332321,-1.72174372401217,-0.594925096592401);
\draw[ultra thin,->] (0,3.5,0)--(-0.41914045800784,-0.302214044561812,-3.05310303077661);
\draw[ultra thin,->] (0,3.5,0)--(-2.9352958614905,-0.503576199821001,0.760227516350117);
\draw[ultra thin,->] (0,3.5,0)--(-1.7384908312755,-1.36721708069111,1.98779152009302);
\draw[ultra thin,->] (0,3.5,0)--(-0.695982212871221,-0.214518239178582,3.22792647077254);
\draw[ultra thin,->] (0,3.5,0)--(2.38522596760973,-0.426592519115469,-1.90528243121436);

\draw[ultra thin,->] (0,-3.5,0)--(1.1*2.10439511361836,0.66388581563,1.1*2.78210283061484);
\draw[ultra thin,->] (0,-3.5,0)--(1.1*3.01787234499297,0.42816653652,1.1*1.77126985984843);
\draw[ultra thin,->] (0,-3.5,0)--(-1.46085266627113, 0.591699377196933,-3.12502821338837);
\draw[ultra thin,->] (0,-3.5,0)--(-2.62000510556625,1.30021364499819,-1.92224288896261);
\draw[ultra thin,->] (0,-3.5,0)--(1.04*-3.4549740779355,0.28927438775 ,1.04*-0.540881114480421);
\draw[ultra thin,->] (0,-3.5,0)--(0.399897373001729,0.917409054976941,3.35386981215918);
\draw[ultra thin,->] (0,-3.5,0)--(1.83883686721049,2.353526971101,1.82471635387193);
\draw[ultra thin,->] (0,-3.5,0)--(3.19307878214232,1.39012071757429,0.349016162381109);
\draw[ultra thin,->] (0,-3.5,0)--(0.92*3.30294341772445,0.16740613968,0.92*-1.05074562563155);

\fill[color=red] (-2.20518201841681,-0.660298067344373,-2.00971890590983) circle(0.05);
\fill[color=red] (-1.87939723683508,-1.80022337020569,-1.32924997062176) circle(0.05);
\fill[color=red] (-0.903072433754454,-2.6586965587787,0.233891420505342) circle(0.05);
\fill[color=red] (0.682000512109353,-1.98437137566493,2.00202037501091) circle(0.05);
\fill[color=red] (1.47873108832973,-0.60183385992793,2.61340136058392) circle(0.05);
\fill[color=red] (0.834173399159874,-1.68352002015628,-2.25642439885711) circle(0.05);
\fill[color=red] (2.87399747167305,-0.581500962487159,0.891596595833803) circle(0.05);
\fill[color=red] (2.29705073332321,-1.72174372401217,-0.594925096592401) circle(0.05);
\fill[color=red] (-0.41914045800784,-0.302214044561812,-3.05310303077661) circle(0.05);
\fill[color=red] (-2.9352958614905,-0.503576199821001,0.760227516350117) circle(0.05);
\fill[color=red] (-1.7384908312755,-1.36721708069111,1.98779152009302) circle(0.05);
\fill[color=red] (-0.695982212871221,-0.214518239178582,3.22792647077254) circle(0.05);
\fill[color=red] (2.38522596760973,-0.426592519115469,-1.90528243121436) circle(0.05);

\fill[color=blue](1.1*2.10439511361836,0.66388581563,1.1*2.78210283061484) circle(0.05);
\fill[color=blue](1.1*3.01787234499297,0.42816653652,1.1*1.77126985984843) circle(0.05);
\fill[color=blue](-1.46085266627113, 0.591699377196933,-3.12502821338837) circle(0.05);
\fill[color=blue](-2.62000510556625,1.30021364499819,-1.92224288896261) circle(0.05);
\fill[color=blue](1.04*-3.4549740779355,0.28927438775 ,1.04*-0.540881114480421) circle(0.05);
\fill[color=blue](0.399897373001729,0.917409054976941,3.35386981215918) circle(0.05);
\fill[color=blue](1.83883686721049,2.353526971101,1.82471635387193) circle(0.05);
\fill[color=blue](3.19307878214232,1.39012071757429,0.349016162381109) circle(0.05);
\fill[color=blue](0.92*3.30294341772445,0.16740613968,0.92*-1.05074562563155) circle(0.05);

\end{tikzpicture}
\caption{The unit disk in Figure~\ref{stretched_billiard} is seen as a subset of the hyperplane $\left\{z=0\right\}$ in $\R^3$. The set of points lying in this disk is mapped onto the southern  or the northern hemisphere of the sphere $\Sph^2$ by stereographic projections. The pole of the projection  depends on whether a given point belongs to the set $\mathcal{N}_1$ (blue points, pole $S$) or $\mathcal{N}_2$ (red points, pole $N$).}
\label{projecting}
\end{subfigure}
\caption{Lifting of a toral sequence to the sphere.}
\label{projstretchedbilliard}
\end{figure}

\section{Dispersion of the multiples of a badly approximable vector modulo one}\label{secdispersion}

In this section, we complete the proof of Theorem~\ref{mainthm} by showing, with the help of the lifting process introduced in Proposition~\ref{lifting},  that there exists a sequence in $\Sph^d$ satisfying the Delone conditions~\eqref{reldens} and~\eqref{unifdiscr}.\\

\paragraph{\textbf{Notation and background specific to this section.}} A lattice $\mathcal{L}$ is a subset of $\R^d$ of the form $\mathcal{L}= \bigoplus_{j=1}^{d}\Z\bm{x}_j$, where $\bm{x}_1, \,\dots\, , \bm{x}_d$ are linearly independent vectors (we thus consider only full rank lattices). The set $\mathcal{L}^*$ stands for the lattice dual to  $\mathcal{L}$; namely, $$\mathcal{L}^*=\left\{\bm{y}\in\R^d\; :\; \forall \bm{x}\in\mathcal{L},\; \bm{x\cdot y}\in\Z\right\}.$$
The length of the shortest nonzero vector in a lattice $\mathcal{L}$ is denoted by $\mu_1(\mathcal{L})$, and the covering radius of $\mathcal{L}$ by $\rho(\mathcal{L})$. Recall that $$\rho(\mathcal{L})\; =\; \sup_{\bm{z}\in\R^d}\;\inf_{\bm{x}\in\mathcal{L}}\;\left\|\bm{z}-\bm{x} \right\|_2.$$
A theorem due to Banaszczyk~\cite[Theorem 2.2]{bana} relates the quantities $\rho(\mathcal{L})$ and $\mu_1(\mathcal{L}^*)$ as follows~:
\begin{equation}\label{bana}
\mu_1(\mathcal{L}^*)\cdot\rho(\mathcal{L})\;\le\;\frac{d}{2}\cdotp
\end{equation}
$\quad$\\

\paragraph{\textbf{Proof of Theorem~\ref{mainthm}.}} In order to establish Theorem~\ref{mainthm}, in view of the existence of maps $\tau_1$ and $\tau_2$ satisfying  conditions~\eqref{condphi} in Proposition~\ref{lifting}, it suffices to show the existence of a sequence $\left(\bm{x}_k\right)_{k\ge 1}$ in $\K_d$ satisfying the separation and density assumptions~\eqref{unifdiscrbis} and \eqref{conddensx_n}.  Such a sequence will be defined from a badly approximable vector $\bm{\alpha}=\left(\alpha_1, \,\dots\, , \alpha_d\right)\in\R^d$. Recall that this means that 
\begin{equation}\label{defbadvec}
\inf_{q\ge 1}  \; q^{1/d}\left\|q\bm{\alpha}\right\|_{\Z^d} >\; 0.
\end{equation}
From Khintchin's Transference Theorem (see~\cite[Chap.~IV, Theorem~5B]{schmidt}), this is equivalent to the property that
\begin{equation}\label{transference}
\inf_{\bm{b}\in\Z^d\backslash\left\{\bm{0}\right\}}\; \left\|\bm{\alpha\cdot b}\right\|\cdot\left\|\bm{b}\right\|_{\infty}^d\;>\; 0.
\end{equation} 
Furthermore, explicit examples of badly approximable vectors are known in any dimension (cf.~\cite[Chap.~2, \S 4]{schmidt}).

 Given $k\ge 1$, set $$\bm{x}_k\;=\; \left\{k\bm{\alpha}\right\}.$$

The separation condition~\eqref{unifdiscrbis} is then easily verified  since, given $k\ge 1$ and $m\in\Z$ such that $1\le |m|\le k^{1-1/n}$, it holds that $$\left\|\bm{x}_{k+m}- \bm{x}_{k}\right\|_2\;\ge\; \left\|\bm{x}_{k+m}- \bm{x}_{k}\right\|_{\Z^d}\; = \;  \left\|m\bm{\alpha}\right\|_{\Z^d}\;\underset{\eqref{defbadvec}}{\gg}\;\frac{1}{|m|^{1/d}}\;\underset{\eqref{elem-1}}{\ge}\; \frac{1}{\sqrt[n]{k}}\cdotp$$

The density condition~\eqref{conddensx_n} will be established by setting $\mathcal{N}_1$ to be the set of even integers and $\mathcal{N}_2$ to be the set of odd integers. To this end, the key result  is the following~:

\begin{prop}\label{propdensbad}
Assume that $\bm{\alpha}\in\R^d$ is a badly approximable vector. Then, there exists a constant $K>0$ such that for any $\epsilon>0$, the reduction modulo $\Z^d$ of the set of points $\left\{q\bm{\alpha}\; :\; 1\le q\le K\epsilon^{-d}\right\}$ is $\epsilon$--dense in $\T^d$. In other words, for any $\epsilon>0$ and any $\bm{y}\in\R^d$, there exists $q\in\llbracket 1, K\epsilon^{-d}\rrbracket$ such that $\left\|q\bm{\alpha}-\bm{y}\right\|_{\Z^d}\le \epsilon$.
\end{prop}

To show that this proposition indeed implies the sought density condition, fix $k\ge 1$ and $\bm{y}\in\overline{\K}_d$. Let furthermore $m_i$ be an integer of the form $m_i=2q+\eta_i$, where $q\ge 1$ and where $\eta_i\in\left\{0,1\right\}$ is chosen so that $k+m_i\in\mathcal{N}_i$.  Applying Proposition~\ref{propdensbad} to the badly approximable vector $2\bm{\alpha}$ with $\epsilon=1/(\sqrt[n]{k})$, one can guarantee the existence of an index $1\le q\le K k^{1-1/n}$ such that $$\left\|2q\bm{\alpha}-\left(\bm{y}-(k+\eta_i)\bm{\alpha}\right)\right\|_{\Z^d}\;=\; \left\|(k+m_i)\bm{\alpha}-\bm{y}\right\|_{\Z^d}\le \epsilon = \frac{1}{\sqrt[n]{k}}\cdotp$$ If the point $\bm{y}$ lies $\epsilon$--away from the boundary of $\K_d$, the above inequality implies 
that $$\left\|(k+m_i)\bm{\alpha}-\bm{y}\right\|_{\Z^d}\; =\; \left\|\bm{x}_{k+m_i}-\bm{y}\right\|_{\infty}\;\le\;\epsilon,$$ whence the claim. If, however, $\bm{y}$ lies in an $\epsilon$--neighbourhood of $\partial\K_d$, apply the previous reasoning to a vector $\bm{y'}$ $\epsilon$--away from $\partial\K_d$ and $\epsilon$--close to $\bm{y}$. Then, $$\left\|\bm{x}_{k+m_i}-\bm{y}\right\|_{2}\;\ll\; \left\|\bm{x}_{k+m_i}-\bm{y'}\right\|_{\infty}+\left\|\bm{y'}-\bm{y}\right\|_{2}\;=\; \left\|\bm{x}_{k+m_i}-\bm{y'}\right\|_{\Z^d}+\left\|\bm{y'}-\bm{y}\right\|_{2}\;\le\; 2\epsilon=\frac{2}{\sqrt[n]{k}},$$ which establishes the claim in this case also and therefore in all cases. 

\begin{proof}[Proof of Proposition~\ref{propdensbad}] 
The proof will be subdivided into five  steps. Throughout, the implicit constants in the Vinogradov notation $\ll, \gg$ and $\asymp$ are allowed to depend on the dimension  $d$ and on the vector $\bm{\alpha}$, but not on any other parameter.

Let $\epsilon>0$ and let
\begin{equation}\label{boundK}
K\ge 2^{d+1}
\end{equation} 
be given. Assume that the reduction modulo $\Z^d$ of the finite sequence $\left(q\bm{\alpha}\right)_{1\le q\le M}$, where $M=K\epsilon^{-d}$, is not $\epsilon$--dense in $\T^d$. The goal is to show that this cannot happen if $K$ is chosen sufficiently large (depending on $\bm{\alpha}$ and $d$ only).\\

\paragraph{\emph{Step 1: introducing the sequence of best approximants to $\bm{\alpha}$}.} Denote by $\left(q_{\nu}\right)_{\nu\ge 1}$ the strictly increasing sequence of the denominators of the best approximants to $\bm{\alpha}$. These are defined as follows~: for all $\nu\ge 1$ and all $q\in\llbracket 1, q_\nu -1\rrbracket$, $$\left\|q_{\nu}\bm{\alpha}\right\|_{\Z^d}\; <\; \left\|q\bm{\alpha}\right\|_{\Z^d}.$$ It then holds that 
\begin{equation}\label{encadr}
\frac{1}{q_{\nu}^{1/d}}\;\ll\; \left\|q_{\nu}\bm{\alpha}\right\|_{\Z^d}\;\le\;\frac{1}{\left(q_{\nu+1}-1\right)^{1/d}}\cdotp
\end{equation} 
Indeed, the first inequality follows from the definition of bad approximability~\eqref{defbadvec}, whereas the second one is a direct consequence of the definition of $q_\nu$ and of Minkowski's Convex Body Theorem (see~\cite[Theorem~III, Appendix~B]{cassels}). In particular, for all $\nu\ge 1$, 
\begin{equation}\label{ordregrandeur}
q_{\nu}\;\asymp\; q_{\nu+1}.
\end{equation}
\\

\paragraph{\emph{Step 2: reducing the problem to a property of density of the multiples of a best ap\-pro\-xi\-mant.}} Consider the integer $\nu\ge 1$ such that 
\begin{equation}\label{encadrement}
q_{\nu}\;\le\; K\epsilon^{-d}\;<\; q_{\nu+1}.
\end{equation}
If $\bm{p}_{\nu}\in\Z^d$ is the integer vector such that $\left\|q_{\nu}\bm{\alpha}\right\|_{\Z^d}=\left\|q_{\nu}\bm{\alpha}-\bm{p}_\nu\right\|_{\infty}$, then, for any integer $1\le q\le  q_{\nu}$, it holds that $$\left\|q\bm{\alpha}-q\frac{\bm{p}_{\nu}}{q_{\nu}}\right\|_{\infty}\;\underset{\eqref{encadr}}{\le}\; \frac{q}{q_{\nu}\left(q_{\nu+1}-1\right)^{1/d}}\;\le\;\left(\frac{2}{q_{\nu+1}}\right)^{1/d}\;\underset{\eqref{encadrement}}{\le}\;\left(\frac{2}{K}\right)^{1/d}\epsilon\;\underset{\eqref{boundK}}{\le}\; \frac{\epsilon}{2}\cdotp\\$$

This implies that in the metric induced by $\left\|\, .\,\right\|_{\Z^d}$, the successive terms of the sequence $\left(q\bm{\alpha}\right)_{1\le q\le q_\nu}$ remain $\epsilon/2$--close to the \sloppy corresponding terms of the sequence $\left(q\bm{\bm{p}_\nu}/q_{\nu}\right)_{1\le q\le q_{\nu}}$. In particular, the reduction modulo $\Z^d$ of the latter sequence cannot be $\epsilon/2$--dense in $\T^d$ since, by assumption, the former sequence is not $\epsilon$--dense in $\T^d$. \\

\paragraph{\emph{Step 3: recasting the problem in terms of the covering radius of a suitable lattice.}} Let $\mathcal{L}_\nu$ be the lattice in $\R^d$ spanned by the vector $\bm{p}_\nu/q_\nu$ and by the elements of the standard basis in $\R^d$. In other words, $$\mathcal{L}_\nu\;=\; \left\{q\frac{\bm{p}_\nu}{q_\nu}+\bm{m}\; :\; q\in\Z,\; \bm{m}\in\Z^d\right\}.$$ Clearly, the projection of $\mathcal{L}_\nu$ on $\T^d$ coincides with the reduction modulo $\Z^d$ of the finite sequence $\left(q\bm{\bm{p}_\nu}/q_{\nu}\right)_{1\le q\le q_{\nu}}$.
The fact that this sequence is not $\epsilon/2$--dense in $\T^d$ then implies that 
\begin{equation}\label{boundcovrad}
\rho(\mathcal{L}_\nu)\; \gg\; \epsilon
\end{equation} 
(we use here the equivalence of norms in $\R^d$, since the covering radius $\rho(\mathcal{L}_\nu)$ is defined in terms of the Euclidean norm).\\

\paragraph{\emph{Step 4: bounding the length of the shortest vector of the dual lattice.}} It is elementary to check that the lattice dual to $\mathcal{L}_\nu$ is $$\mathcal{L}_\nu^*\;=\;\left\{\bm{v}\in\Z^d\; :\; \bm{v\cdot p}_\nu \equiv 0 \pmod{q_\nu}\right\}.$$ From Banaszczyk's inequality~\eqref{bana}, the lower bound~\eqref{boundcovrad} for the covering radius of the lattice $\mathcal{L}_\nu$ implies the following upper bound for the shortest nonzero vector in the dual lattice $\mathcal{L}_\nu^*$~: $$\mu_1( \mathcal{L}_\nu^*)\;\ll\;\epsilon^{-1}.$$ In other words, there exists $\bm{v}\in\Z^d$ such that 
\begin{equation}\label{shorvector}
1\;\le\; \left\|\bm{v}\right\|_2\;\ll\; \epsilon^{-1}\qquad\textrm{ and }\qquad \bm{v\cdot}\frac{\bm{p}_\nu}{q_\nu}\in\Z.
\end{equation}

\paragraph{\emph{Step 5: conclusion.}}  Let $\bm{v}\in\Z^d$ be the nonzero vector appearing in~\eqref{shorvector} and let $c$ be the integer such that $\bm{v\cdot p}_\nu=cq_\nu$.  Then, using the Cauchy--Schwarz inequality to obtain the relation in the second line below, one has~:
\begin{align*}
\epsilon^d\;\ll\; \left\|\bm{v}\right\|_{\infty}^{-d}\; \underset{\eqref{transference}}{\ll}\; \left\|\bm{\alpha\cdot v}\right\|\; & \le\; \left|\bm{\alpha\cdot v}-c\right|\;=\; \left|\bm{v\cdot}\left( \bm{\alpha}- \frac{\bm{p}_{\nu}}{q_\nu}\right)\right|\\
&\le\; \left\|\bm{v}\right\|_2 \cdot\left\| \bm{\alpha}- \frac{\bm{p}_{\nu}}{q_\nu}\right\|_2 \\
&\underset{\eqref{encadr}\&\eqref{shorvector}}{\ll}\; \frac{\epsilon^{-1}}{q_{\nu}q_{\nu+1}^{1/d}}\\
&\underset{\eqref{ordregrandeur}\&\eqref{encadrement}}{\ll}\; \frac{\epsilon^d}{K^{1+1/d}}\cdotp
\end{align*}
This leads to a contradiction for $K$ large enough, and thus completes the proof.
\end{proof}

\section{A Focus on the three--dimensional Case}\label{sec3}

In this section, we provide, in the case of the three dimensional Euclidean space, an alternative to the construction of a  spiral Delone set presented in Sections~\ref{liftsection} and~\ref{secdispersion}. That construction was based on the lifting of a linear toral sequence to the sphere. Here, we show that a spiral Delone set can be obtained in $\R^3$ with the help of a  sequence in $\Sph^2$ arising from a continuous spherical flow.\\

Let $\mathcal{T}$ be a (solid) regular tetrahedron with centroid at the origin admitting the unit sphere $\Sph^2$ as a circumscribing sphere. Denote by 
$A, B, C$ and $D$  its vertices labelled as in Figures~\ref{regtetra} and~\ref{tetranet}.  From elementary geometric considerations, $\mathcal{T}$ has height $h=4/3$ and edge length $a=\sqrt{8/3}$. 

Given a vector $\bm{\alpha}=\left(\alpha_1, \alpha_2\right)\in\R^2$ such that $\alpha_1, \alpha_2> 0$, consider the geodesic flow on the surface of $\mathcal{T}$ in direction $\bm{\alpha}$ starting at a given vertex (say, $B$) on a given face containing this vertex (say, $ABC$). It is defined as the map
\begin{equation}\label{flowtetra}
\mathfrak{t}_{\bm{\alpha}}~: t\ge 0\;\mapsto\; \mathfrak{t}_{\bm{\alpha}}\left(t\right)
\end{equation}
such that 
\begin{equation}\label{equaflowtetra}
\mathfrak{t}_{\bm{\alpha}}(t)\;=\; t\alpha_1\cdotp \overrightarrow{BC}+t\alpha_2\cdotp \overrightarrow{BA}
\end{equation}
as long as the right--hand side remains on the face $ABC$ (here$, BC$ and $BA$ are the two edges on the face $ABC$ adjacent to the vertex $B$, taken counterclockwise). Whenever this line hits an edge on $\partial \mathcal{T}$, it is prolonged by a process of unfolding~: the faces adjacent to the edge are rotated along this edge so as to be in coplanar position; the line is then extended in the natural way to the adjacent face. All faces are then turned back to their initial positions. Repeating this process yields a well--defined linear flow on $\partial \mathcal{T}$ such as the one represented in Figure~\ref{regtetra}, provided that the flow does not hit a vertex. This situation can easily been avoided upon requiring that the ratio $\alpha_1/\alpha_2$ is irrational, which we shall assume henceforth.

If 
\begin{equation}\label{radialproj}
\mathfrak{r}~: \bm{x}\in\R^3\backslash\left\{\bm{0}\right\}\;\mapsto\; \frac{\bm{x}}{\left\|\bm{x}\right\|_2}\in\Sph^2
\end{equation}
denotes the radial projection onto  the sphere, the map 
\begin{equation}\label{geodesictetra}
t\ge 0\;\mapsto\; \left(\mathfrak{r}\circ\mathfrak{t}_{\bm{\alpha}}\right) (t)
\end{equation} 
then defines a continuous flow on $\Sph^2$. It is represented in Figure~\ref{sphtetra}. 

As will be clear from the properties of this flow established below, the curve defined by~\eqref{geodesictetra} necessarily self--intersects. This prevents the spherical sequence $\left(\left(\mathfrak{r}\circ\mathfrak{t}_{\bm{\alpha}}\right) (k)\right)_{k\ge 1}$ from satisfying the separation condition~\eqref{unifdiscr} needed for a spiral set defined from it to be Delone. The following theorem nevertheless shows that a subsequence of this spherical sequence can be used to construct such a Delone set. This subsequence should be seen as being defined in a natural way in view of the equivalent reformulation of the separation condition in terms of a packing property in~\eqref{packsep}.

\begin{thm}\label{thmdistrtetra}
Assume that $\bm{\alpha}=\left(\alpha_1, \alpha_2\right)\in\R^2$ is a badly approximable vector such that $\alpha_1, \alpha_2>0$. Let $\gamma>0$. Define by induction a strictly increasing sequence of natural integers $\left(j_k\right)_{k\ge 1}$ and a spherical sequence $\left(\bm{u}_k\right)_{k\ge 1}$ as follows~: set $j_1=1$ and $\bm{u}_1=\left(\mathfrak{r}\circ\mathfrak{t}_{\bm{\alpha}}\right) (j_1)$. 

For $k\ge 2$, let $\mathfrak{D}_k$ be the subset of the sphere $\Sph^2$ given by the union $$\mathfrak{D}_k\;=\;\bigcup_{1\le m < k^{2/3}}\mathcal{C}_2\left(\bm{u}_{k-m}, \, \frac{\gamma}{k^{1/3}}\right).$$  The integer $j_k$ and the unit vector $\bm{u}_k$ are defined from $\mathfrak{D}_k$ by setting 
\begin{equation}\label{defj_ku_k}
j_k\;=\;\min\left\{p>j_{k-1}\; :\; \left(\mathfrak{r}\circ\mathfrak{t}_{\bm{\alpha}}\right) (p)\not\in\mathfrak{D}_k\right\}\in\N\qquad\textrm{ and }\qquad \bm{u}_k=\left(\mathfrak{r}\circ\mathfrak{t}_{\bm{\alpha}}\right) (j_k)\in\Sph^2.
\end{equation} 

Then, for  $\gamma>0$ small enough, the integer sequence $\left(j_k\right)_{k\ge 1}$ is well--defined. Moreover, the spiral set~\eqref{sk} obtained from the above spherical sequence $\left(\bm{u}_k\right)_{k\ge 1}$ is Delone in $\R^3$.
\end{thm} 

In preparation for the proof of Theorem~\ref{thmdistrtetra}, we first show how  to reduce the study of the tetrahedral flow~\eqref{flowtetra} to that of a toral flow. 

The net of the tetrahedron (Figure~\ref{tetranet}) can be used to tile the plane in a way that is compatible with the identification of the edges, as illustrated in  Figure~\ref{tiling tetra}~: two triangles represent  the same face of the tetrahedron (hence, have the same label) if, and only if, they share a unique common vertex. The process of unfolding then maps the geodesic flow~\eqref{flowtetra} to a half--line $(\Delta)$ as represented  in Figure~\ref{tiling tetra}. This half--line has parametric equation~\eqref{equaflowtetra} if one identifies $B$ with the origin of $(\Delta)$ in Figure~\ref{tiling tetra} and $A$ and $C$ with the vertices of the corresponding triangle 1 when labelled as in Figure~\ref{tetranet}.

Consider the matrix $$\mathfrak{M}= 
\begin{pmatrix}
\sqrt{3/8}&-3/8\\
0 & 3/4
\end{pmatrix},
$$
which maps the vector $\overrightarrow{BC}=\left(\sqrt{8/3}, 0\right)^T$ to the vector $\left(1, 0\right)^T$ and the vector $\overrightarrow{BA}=\left(\sqrt{2/3}, 4/3\right)^T$  to the vector $\left(0, 1\right)^T$.  The image under the linear transformation $\mathfrak{M}$ of the vertices of the tiling obtained from the net of the tetrahadron (Figure~\ref{tiling tetra}) is then the standard lattice in Figure~\ref{biliptiling}. This linear transformation maps the half--line with parametric equation~\eqref{equaflowtetra} to the half--line with parametric equation $t\mapsto t\bm{\alpha}$.  By periodicity, this half--line projects onto a linear flow in the $2\times 2$ torus identified with the square $\K'_2=[-1,1)^2$, namely
\begin{equation}\label{eqlinflow}
\tilde{\mathfrak{s}}_{\bm{\alpha}}~: t\;\mapsto\; \left(2\left\{t\alpha_1+\frac{1}{2}\right\}-1, \;2\left\{t\alpha_2+\frac{1}{2}\right\}-1\right).
\end{equation}
This is represented in Figure~\ref{reductoralflow}. \\

Thus, upon unfolding the tetrahedral flow~\eqref{flowtetra}, taking its image under the linear map $\mathfrak{M}$ and then reducing it by periodicity, one obtains the toral flow~\eqref{eqlinflow}. Not all steps in this process are, however, bi--Lipschitz, in such a way that that the density and separation properties of a toral sequence obtained from~\eqref{eqlinflow} cannot be immediately lifted to the  tetrahedral sequence it arises from.  

To see this, note first that the square $\K'_2$ contains two copies of each of the faces of the tetrahedron $\mathcal{T}$ (see Figure~\ref{reductoralflow}). Two faces sharing the same label are related by a reflection in the origin. When considering the tetrahedral flow~\eqref{flowtetra}, this is translated by the fact that this flow crosses each face in two different directions, and therefore self--intersects on each face. 

Another reason implying that the separation condition can fail to transfer from the toral flow to the surface of the tetrahedron is when folding back the net. Indeed, it should be clear that the distance between two points lying in the same triangle in the square $\K'_2$ is, up to a multiplicative constant depending only on the bi--Lipschitz map induced by $\mathfrak{M}$, the same as the distance between the two points they come from in the corresponding face of the tetrahedron. As the unfolding process distorts distances between faces, this, however, need not be true anymore if the two points lie in two different triangles in $\K'_2$. 

As will be made clear from the proof of Theorem~\ref{thmdistrtetra} below, the definition of the sequence $\left(j_k\right)_{k\ge 1}$ in~\eqref{defj_ku_k} provides a natural way to avoid the above--mentioned obstacles to  fulfil the sought separation conditon.\\

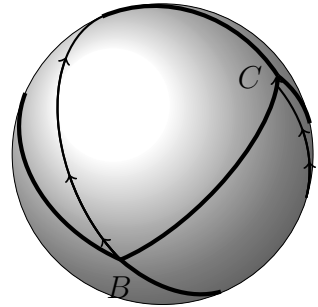
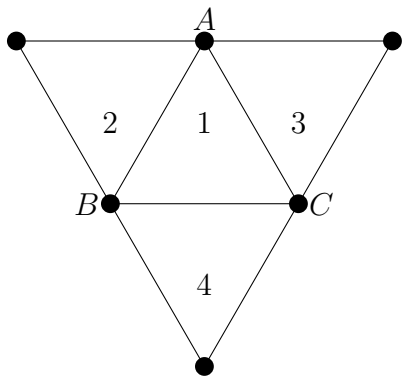
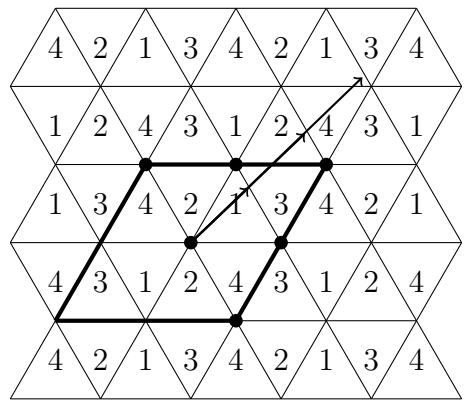
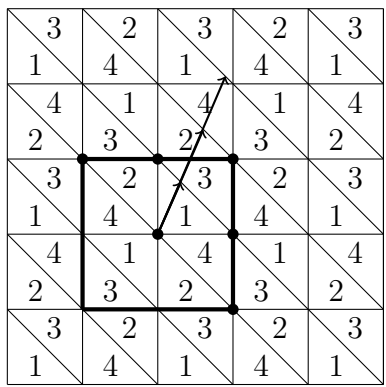
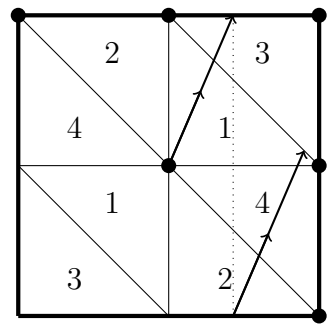
\begin{figure}[h!]
\begin{subfigure}{.45\textwidth}
  \centering
\begin{tikzpicture}[scale=1.8]

\draw[thick] (1,0, -0.7071) -- (0,1,0.7071);
\draw[thick]    (-1,0,-0.7071)-- (0,1, 0.7071);
\draw[thick]    (0,-1,0.7071) -- (0,1, 0.7071);
\draw[thick]    (0,-1,0.7071) -- (-1,0,-0.7071);
\draw[dashed,thick]  (-1,0,-0.7071) -- (1,0, -0.7071);
\draw[thick]   (0,-1,0.7071) -- (1,0, -0.7071);

\draw (0,-1,0.7071) -- (0.33,0.66, 0.33*0.7071 );
\draw[->] (0,-1,0.7071) -- (0.1667,-0.1667, 0.6667*0.7071 );

\draw[dashed] (0.33,0.66, 0.33*0.7071 )--(-0.5, 0.5, 0);
\draw[dashed, ->] (0.33,0.66, 0.33*0.7071 )--(0.5*-0.5+0.5*0.33, 0.5*0.5+0.5*0.66, 0.5*0.33*0.7071);

\draw (-0.5, 0.5, 0)--(-0.66, -0.33, -0.33*0.7071);
\draw[->] (-0.5, 0.5, 0)--(0.5*-0.66+0.5*-0.5, 0.5*-0.33+0.5*0.5, 0.5*-0.33*0.7071);

\draw[dashed] (-0.66, -0.33, -0.33*0.7071)-- (0.9*1,0.1*-1, 0.9*-0.7071+0.1*0.7071);
\draw[dashed, ->] (-0.66, -0.33, -0.33*0.7071)-- (0.5*0.9*1+0.5*-0.66,0.5*0.1*-1+0.5*-0.33, 0.5*0.9*-0.7071+0.5*0.1*0.7071+0.5*-0.33*0.7071);

\draw[->] (0.9*1,0.1*-1, 0.9*-0.7071+0.1*0.7071)--(0.5*0.95+0.5*0.9*1, 0.5*0.05+0.5*0.1*-1,0.5*-0.9*0.7071+0.5*0.9*-0.7071+0.5*0.1*0.7071);

\draw (0,1,0.7071) node[anchor=south]   {$A$};
\draw (0,-1,0.7071) node[anchor=north]   {$B$};
\draw (1,0, -0.7071) node[anchor=west]   {$C$};
\draw (-1,0,-0.7071) node[anchor=east]   {$D$};

\end{tikzpicture}
  \caption{Linear flow on  the surface of a regular tetrahedron $ABCD$.}
  \label{regtetra}
\end{subfigure}
\hfill
\begin{subfigure}{.45\textwidth}
  \centering
\begin{tikzpicture}[scale=1]
       \tdplotsetrotatedcoords{20}{80}{0}
       \draw [ball color=white,very thin,tdplot_rotated_coords] (0,0,0) circle (2) ;

\fill (1.2,0.75, -0.7071)  node[anchor=east]   {$C$};
\fill (0,-0.9,1.5)  node[anchor=north]   {$B$};

        \tdplotdefinepoints(0,0,0)(0.75*0+0.25*1,0.75*-1+0.25*0,0.75*0.7071+0.25*-0.7071)(1,0,-0.7071)
	       \tdplotdrawpolytopearc[ultra thick]{2}{}{}

        \tdplotdefinepoints(0,0,0)(0, -0.333, 0.7071)(0,1,0.7071)
	       \tdplotdrawpolytopearc[ultra thick]{2}{}{}

        \tdplotdefinepoints(0,0,0)(0,1,0.7071)(1,0, -0.7071)
	       \tdplotdrawpolytopearc[ultra thick]{2}{}{}

        \tdplotdefinepoints(0,0,0)(1,0, -0.7071)(0.55*1+0.45*-1,0.55*0+0.45*0.8165,0.55*-0.7071+0.45*-0.7071)
	       \tdplotdrawpolytopearc[ultra thick]{2}{}{}

        \tdplotdefinepoints(0,0,0)(0.35*-1+0.65*0,0.35*0.8165+0.65*1,0.65*0.7071+0.35*-0.7071)(0,1,0.7071)
	       \tdplotdrawpolytopearc[ultra thick]{2}{}{}

        \tdplotdefinepoints(0,0,0)(1,0, -0.7071)(0, -0.333, 0.7071)
	       \tdplotdrawpolytopearc[thick]{2}{}{}

        \tdplotdefinepoints(0,0,0)(1,0, -0.7071)(0.25*0+0.75*1, 0.25*-0.333+0.75*0, 0.25*0.7071+0.75*-0.7071)
	       \tdplotdrawpolytopearc[thick, ->]{2}{}{}

        \tdplotdefinepoints(0,0,0)(1,0, -0.7071)(0.5*0+0.5*1, 0.5*-0.333+0.5*0, 0.5*0.7071+0.5*-0.7071)
	       \tdplotdrawpolytopearc[thick, ->]{2}{}{}

        \tdplotdefinepoints(0,0,0)(1,0, -0.7071)(0.75*0+0.25*1, 0.75*-0.333+0.25*0, 0.75*0.7071+0.25*-0.7071)
	       \tdplotdrawpolytopearc[thick, ->]{2}{}{}
	       	       	       
        \tdplotdefinepoints(0,0,0)(0.6*-0.33+0.4*0.1, 0.6*0.54+0.4*0.9, 0.6*-0.71+0.4*0.9*0.7071+0.4*0.1*-0.7071)(0.1, 0.9, 0.9*0.7071+0.1*-0.7071)	 
	       \tdplotdrawpolytopearc[thick]{2}{}{}

        \tdplotdefinepoints(0,0,0)(0.6*-0.33+0.4*0.1, 0.6*0.54+0.4*0.9, 0.6*-0.71+0.4*0.9*0.7071+0.4*0.1*-0.7071)(0.75*0.6*-0.33+0.75*0.4*0.1+0.25*0.1,0.75*0.6*0.54+0.75*0.4*0.9 +0.25*0.9, 0.75*0.6*-0.71+0.75*0.4*0.9*0.7071+0.75*0.4*0.1*-0.7071+0.25*0.9*0.7071+0.25*0.1*-0.7071)	 
	       \tdplotdrawpolytopearc[thick, ->]{2}{}{}

        \tdplotdefinepoints(0,0,0)(0.6*-0.33+0.4*0.1, 0.6*0.54+0.4*0.9, 0.6*-0.71+0.4*0.9*0.7071+0.4*0.1*-0.7071)(0.5*0.6*-0.33+0.5*0.4*0.1+0.5*0.1,0.5*0.6*0.54+0.5*0.4*0.9 +0.5*0.9, 0.5*0.6*-0.71+0.5*0.4*0.9*0.7071+0.5*0.4*0.1*-0.7071+0.5*0.9*0.7071+0.5*0.1*-0.7071)	 
	       \tdplotdrawpolytopearc[thick, ->]{2}{}{}
	       
        \tdplotdefinepoints(0,0,0)(0.1, 0.9, 0.9*0.7071+0.1*-0.7071)(0.5*0.1+0.5*0, 0.5*0.9+0.5*0.95*1+0.5*0.05*-1, 0.5*0.9*0.7071+0.5*0.1*-0.7071+0.5*0.95*0.7071+0.5*0.05*0.7071)	 
	       \tdplotdrawpolytopearc[thick, ->]{2}{}{}	  
	       	       	       	       	       	       
\end{tikzpicture}
  \caption{The spherical tetrahedron obtained by radial  projection of the terahedron $ABCD$, and the corresponding spherical flow.}
  \label{sphtetra}
\end{subfigure}
\newline
\vspace{6mm}
\begin{subfigure}{.45\textwidth}
  \centering
  \begin{tikzpicture}[scale=1.25]
\draw (0,0)-- (-2, 2*1.732);
\draw (0,0)-- (2, 2*1.732);
\draw (-1,1.732)-- (1,1.732);
\draw (-1,1.732)-- (0,2*1.732);
\draw (1,1.732)-- (0,2*1.732);
\draw (-2,2*1.732)-- (2,2*1.732);

\draw (0,2*1.732) node[anchor=south]   {$A$};
\draw (-1,1.732) node[anchor=east]   {$B$};
\draw (1,1.732) node[anchor=west]   {$C$};

\fill (0,2*1.732) circle (0.1) ;
\fill (-1,1.732)  circle (0.1) ;
\fill (1,1.732) circle (0.1) ;
\fill (2,2*1.732) circle (0.1) ;
\fill (0,0)  circle (0.1) ;
\fill (-2,2*1.732) circle (0.1) ;

\draw (0,1.5*1.732) node   {$1$};
\draw (0,0.5*1.732) node   {$4$};
\draw (-1,1.5*1.732) node   {$2$};
\draw (1,1.5*1.732) node   {$3$};
\end{tikzpicture}
  \caption{Net of the tetrahedron $ABCD$ and labelling of the faces.}
  \label{tetranet}
\end{subfigure}
\hfill
\begin{subfigure}{.45\textwidth}
  \centering
\vspace{8mm }
\begin{tikzpicture}[scale=0.6]

\draw (-6,0)-- (4, 0);
\draw (-5,1.732)-- (3, 1.732);
\draw (-5,-1.732)-- (3, -1.732);
\draw (-6,2*1.732)-- (4, 2*1.732);
\draw (-6,-2*1.732)-- (4, -2*1.732);
\draw (-5,3*1.732)--(3, 3*1.732);

\draw (2,-2*1.732)-- (-2, 2*1.732);
\draw (-2,-2*1.732)-- (2, 2*1.732);

\draw (0,2*1.732)-- (-4,-2*1.732);
\draw (0,-2*1.732)-- (4, 2*1.732);

\draw (2,-2*1.732)-- (4, 0);

\draw (0,2*1.732)-- (4,-2*1.732);

\draw (4, 0)-- (2, 2*1.732);

\draw (0,-2*1.732)-- (-4,2*1.732);

\draw (-2,-2*1.732)-- (-6,2*1.732);
\draw (-4,-2*1.732)-- (-6,0);

\draw (-6,2*-1.732)--(-2, 2*1.732);
\draw (-6,0)--(-4, 2*1.732);

\draw (0,2*1.732)--(-1, 3*1.732);
\draw (-2,2*1.732)--(-3, 3*1.732);
\draw (-4,2*1.732)--(-5, 3*1.732);
\draw (2,2*1.732)--(1, 3*1.732);
\draw (4,2*1.732)--(3, 3*1.732);

\draw (0,2*1.732)--(1, 3*1.732);
\draw (2,2*1.732)--(3, 3*1.732);
\draw (-2,2*1.732)--(-1, 3*1.732);
\draw (-4,2*1.732)--(-3, 3*1.732);
\draw (-6,2*1.732)--(-5, 3*1.732);

\draw[ultra thick] (-3,1.732)--(1, 1.732);
\draw[ultra thick] (-1,-1.732)--(-5, -1.732);
\draw[ultra thick] (1, 1.732)--(-1, -1.732);
\draw[ultra thick] (-3,1.732)--(-5,-1.732);

\fill (-2, 0) circle (0.15) ;
\fill (-1, 1.732) circle (0.15) ;
\fill (0, 0) circle (0.15) ;
\fill (-3, 1.732) circle (0.15) ;
\fill (1, 1.732) circle (0.15) ;
\fill (-1, -1.732) circle (0.15) ;

\draw (0,1.5*1.732) node   {$2$};
\draw (0,0.5*1.732) node   {$3$};
\draw (-1,1.5*1.732) node   {$1$};
\draw (1,1.5*1.732) node   {$4$};

\draw (0,-1.5*1.732) node   {$2$};
\draw (0,-0.5*1.732) node   {$3$};
\draw (-1,0.5*1.732) node   {$1$};
\draw (1,0.5*1.732) node   {$4$};

\draw (-2,0.5*1.732) node   {$2$};
\draw (2,-1.5*1.732) node   {$3$};
\draw (1,-0.5*1.732) node   {$1$};
\draw (-1,-0.5*1.732) node   {$4$};

\draw (-2,-0.5*1.732) node   {$2$};
\draw (-2,-1.5*1.732) node   {$3$};
\draw (-3,-0.5*1.732) node   {$1$};
\draw (-1,-1.5*1.732) node   {$4$};

\draw (-4,-1.5*1.732) node   {$2$};
\draw (-4,-0.5*1.732) node   {$3$};
\draw (1,-1.5*1.732) node   {$1$};
\draw (-3,0.5*1.732) node   {$4$};

\draw (-4,1.5*1.732) node   {$2$};
\draw (-4,0.5*1.732) node   {$3$};
\draw (-3,-1.5*1.732) node   {$1$};
\draw (-3,1.5*1.732) node   {$4$};

\draw (2,-0.5*1.732) node   {$2$};
\draw (-2,1.5*1.732) node   {$3$};
\draw (3,0.5*1.732) node   {$1$};
\draw (3,-0.5*1.732) node   {$4$};

\draw (2,0.5*1.732) node   {$2$};
\draw (2,1.5*1.732) node   {$3$};
\draw (3,1.5*1.732) node   {$1$};
\draw (3,-1.5*1.732) node   {$4$};

\draw (-5,1.5*1.732) node   {$1$};
\draw (-5,0.5*1.732) node   {$1$};
\draw (-5,-0.5*1.732) node   {$4$};
\draw (-5,-1.5*1.732) node   {$4$};
\draw (-5,2.5*1.732) node   {$4$};

\draw (-4,2.5*1.732) node   {$2$};
\draw (-3,2.5*1.732) node   {$1$};
\draw (-2,2.5*1.732) node   {$3$};
\draw (-1,2.5*1.732) node   {$4$};

\draw (0,2.5*1.732) node   {$2$};
\draw (1,2.5*1.732) node   {$1$};
\draw (2,2.5*1.732) node   {$3$};
\draw (3,2.5*1.732) node   {$4$};


\draw[thick, ->] (-2,0)--(1.8, 2*1.732+0.2);
\draw[thick, ->] (-2,0)--(0.66*-2+0.33*1.8, 0.33*2*1.732+0.33*0.2);
\draw[thick, ->] (-2,0)--(0.33*-2+0.66*1.8, 0.66*2*1.732+0.66*0.2);
\end{tikzpicture}
  \caption{The net of the tetrahedron can be used to tile the plane in a way that is compatible with the identification of the edges. By unfolding, the geodesic flow on the surface of the tetrahedron becomes a straight line in $\R^2$.}
  \label{tiling tetra}
\end{subfigure}
\newline
\vspace{6mm}
\begin{subfigure}{.45\textwidth}
  \centering
\begin{tikzpicture}[scale=0.5]

\fill (0,0) circle (0.15) ;
\fill (-4,2) circle (0.15) ;
\fill (0,-2) circle (0.15) ;
\fill (-2,0) circle (0.15) ;
\fill (0,2) circle (0.15) ;
\fill (-2,2) circle (0.15) ;

\draw (-6,-4)--(-6,6);
\draw (0,-4)--(0,6);
\draw (-4,-4)--(-4, 6);
\draw (-2,-4)--(-2, 6);
\draw (2,-4)--(2, 6);
\draw (4,-4)--(4, 6);
\draw (4,-4)--(4, 6);

\draw (-6,0)--(4,0);
\draw (-6,2)--(4,2);
\draw (-6,4)--(4,4);
\draw (-6,6)--(4,6);
\draw (-6,-2)--(4,-2);
\draw (-6,-4)--(4,-4);

\draw (4,-4)--(-6,6);
\draw (4,-2)--(-4,6);
\draw (4,0)--(-2,6);
\draw (4,2)--(0,6);
\draw (4,4)--(2,6);

\draw (2,-4)--(-6,4);
\draw (0,-4)--(-6,2);
\draw (-2,-4)--(-6,0);
\draw (-4,-4)--(-6,-2);

\draw[ultra thick] (-4,2)--(-4,-2);
\draw[ultra thick] (-4,2)--(0,2);
\draw[ultra thick] (0,2)--(0,-2);
\draw[ultra thick] (0,-2)--(-4,-2);

\draw[thick, ->] (-2,0)--(-0.2,4.2);
\draw[thick, ->] (-2,0)--(0.33*-0.2+0.66*-2,0.33*4.2);
\draw[thick, ->] (-2,0)--(0.66*-0.2+0.33*-2,0.66*4.2);

\draw (-1.25,0.5) node   {$1$};
\draw (0.75,0.5) node   {$4$};
\draw (2.75,0.5) node   {$1$};
\draw (-3.25,0.5) node   {$4$};
\draw (-5.25,0.5) node   {$1$};

\draw (-1.25,-3.5) node   {$1$};
\draw (0.75,-3.5) node   {$4$};
\draw (2.75,-3.5) node   {$1$};
\draw (-3.25,-3.5) node   {$4$};
\draw (-5.25,-3.5) node   {$1$};

\draw (-1.25,4.5) node   {$1$};
\draw (0.75,4.5) node   {$4$};
\draw (2.75,4.5) node   {$1$};
\draw (-3.25,4.5) node   {$4$};
\draw (-5.25,4.5) node   {$1$};

\draw (-1.25,-1.5) node   {$2$};
\draw (0.75,-1.5) node   {$3$};
\draw (2.75,-1.5) node   {$2$};
\draw (-3.25,-1.5) node   {$3$};
\draw (-5.25,-1.5) node   {$2$};

\draw (-1.25,2.5) node   {$2$};
\draw (0.75,2.5) node   {$3$};
\draw (2.75,2.5) node   {$2$};
\draw (-3.25,2.5) node   {$3$};
\draw (-5.25,2.5) node   {$2$};

\draw (-4.75,1.5) node   {$3$};
\draw (-2.75,1.5) node   {$2$};
\draw (-0.75,1.5) node   {$3$};
\draw (1.25,1.5) node   {$2$};
\draw (3.25,1.5) node   {$3$};

\draw (-4.75,-2.5) node   {$3$};
\draw (-2.75,-2.5) node   {$2$};
\draw (-0.75,-2.5) node   {$3$};
\draw (1.25,-2.5) node   {$2$};
\draw (3.25,-2.5) node   {$3$};

\draw (-4.75,-0.5) node   {$4$};
\draw (-2.75,-0.5) node   {$1$};
\draw (-0.75,-0.5) node   {$4$};
\draw (1.25,-0.5) node   {$1$};
\draw (3.25,-0.5) node   {$4$};

\draw (-4.75,5.5) node   {$3$};
\draw (-2.75,5.5) node   {$2$};
\draw (-0.75,5.5) node   {$3$};
\draw (1.25,5.5) node   {$2$};
\draw (3.25,5.5) node   {$3$};

\draw (-4.75,3.5) node   {$4$};
\draw (-2.75,3.5) node   {$1$};
\draw (-0.75,3.5) node   {$4$};
\draw (1.25,3.5) node   {$1$};
\draw (3.25,3.5) node   {$4$};

\end{tikzpicture}
  \caption{Upon applying a linear transformation, the previous tiling is bi--Lipschitz to the standard lattice.}
  \label{biliptiling}
\end{subfigure}
\hfill
\begin{subfigure}{.45\textwidth}
  \centering
\begin{tikzpicture}[scale=1]

\draw (-1.25,0.5) node{$4$};
\draw (-0.75,1.5) node{$2$};
\draw (1.25,1.5) node{$3$};
\draw (0.75,0.5) node{$1$};
\draw (1.25,-0.5) node{$4$};
\draw (-0.75,-0.5) node{$1$};
\draw (-1.25,-1.5) node{$3$};
\draw (0.75,-1.5) node{$2$};

\draw[ultra thick] (-2,-2)--(-2,2);
\draw[ultra thick] (-2,2)--(2,2);
\draw[ultra thick] (2,2)--(2,-2);
\draw[ultra thick] (2,-2)--(-2,-2);

\draw[thick, ->] (0, 0)--(0.857, 2);
\draw[thick, ->] (0, 0)--(0.5*0.857, 0.5*2);
\draw[thick, ->] (0.857, -2)--(1.8, 0.2);
\draw[thick, ->] (0.857, -2)--(0.5*1.8+0.5*0.857, 0.5*0.2+0.5*-2);
\draw[dotted] (0.857, -2)-- (0.857, 2);

\draw (-2,2)--(2,-2);
\draw(-2,0)--(0,-2);
\draw(0,2)--(2,0);
\draw(-2,0)--(2,0);
\draw(0,-2)--(0,2);

\fill (0,0) circle (0.1) ;
\fill (2,2) circle (0.1) ;
\fill (2,-2) circle (0.1) ;
\fill (2,0) circle (0.1) ;
\fill (0,2) circle (0.1) ;
\fill (-2,2) circle (0.1) ;

\end{tikzpicture}
  \caption{By periodicity, the flow on the surface of the tetrahedron $ABCD$ reduces to a toral flow.}
  \label{reductoralflow}
\end{subfigure}
\vspace{-8mm}
\caption{Study of a linear flow on the surface of a tetrahedron}
\label{flowtetratheory}
\end{figure}

\begin{proof}[Proof of Theorem~\ref{thmdistrtetra}]
In view of the equivalent formulation of bad approximability given by condition~\eqref{transference}, the flow~\eqref{geodesictetra} is well--defined under the assumption that the vector $\bm{\alpha}$ is badly approximable, since the ratio $\alpha_1/\alpha_2$ can then not be rational.\\

The density and separation conditions~\eqref{reldens} and~\eqref{unifdiscr} of the spherical sequence $\left(\bm{u}_k\right)_{k\ge 1}$ defined in~\eqref{defj_ku_k} will now be established in four steps.\\

\paragraph{\emph{Step 1:  the radial projection map~\eqref{radialproj} restricted to $\partial\mathcal{T}$ is bi--Lipschitz.}} To see that this indeed holds,  let $\bm{x}=\rho(\bm{u})\bm{u}\in\partial\mathcal{T}$ and $\bm{y}=\rho(\bm{v})\bm{v}\in\partial\mathcal{T}$, where $\bm{u}, \bm{v}\in\Sph^2$ and $\rho(\bm{u}), \rho(\bm{v})>0$. Then, 
\begin{equation}\label{lipradial}
\left\|\bm{x}-\bm{y}\right\|_2 \;\underset{\eqref{elem3}}{\asymp} \; \left|\rho(\bm{v})-\rho(\bm{u})\right|+\sqrt{\rho(\bm{v})\rho(\bm{u})}\cdot\dsph(\bm{u},\bm{v})\;\asymp\; \dsph(\bm{u},\bm{v}) \;=\; \dsph(\mathfrak{r}(\bm{x}),\mathfrak{r}(\bm{y})).
\end{equation}
The second relation above follows from two observations~: on the one hand, since the origin is an interior point of the terahedron $\mathcal{T}$, there exists $c>0$ such that $c\le \rho(\bm{v}), \rho(\bm{u})\le 1$. On the other, the map $\rho~: \Sph^2\rightarrow \R$ which to $\bm{u}\in\Sph^2$ associates the positive real number $\rho(\bm{u})>0$ such that $\rho(\bm{u})\bm{u}\in\partial\mathcal{T}$ is Lipschitz. This can be inferred from a direct calculation in the particular case under consideration, or else be seen as a general property of convex bodies containing the origin in their interior, as stated in~\cite[Lemma 3.1]{radialconvex}.

This step implies that it is enough to verify the  density and separation conditions~\eqref{reldens} and~\eqref{unifdiscr}  for the \emph{tetrahedral} sequence $\left(\bm{t_{\alpha}}\left(j_k\right)\right)_{k\ge 1}$.\\

\paragraph{\emph{Step 2:  the sequence $\left(j_k\right)_{k\ge 1}$ is well--defined provided that the parameter $\gamma>0$ is small enough.}} Clearly, given any $k\ge 1$, $$\sigma_2\left(\mathfrak{D}_k\right)\;\ll\; \gamma^2,$$ where the implicit constant can be taken as absolute. Therefore, for $\gamma>0$ small enough, $\sigma_2\left(\mathfrak{D}_k\right)<\sigma_2\left(\Sph^2\right)$. Since $\mathfrak{D}_k$ is a finite union of spherical caps, the existence of $j_k$ follows from the density of the sequence $\left(\left(\mathfrak{r}\circ\mathfrak{t}_{\bm{\alpha}}\right)(p)\right)_{p>j_{k-1}}$. Since the radial projection map $\mathfrak{r}$ is continuous, this, in turn, is implied by the density of the sequence $\left(\mathfrak{t}_{\bm{\alpha}}(p)\right)_{p>j_{k-1}}$ on $\partial\mathcal{T}$. 

To establish the latter claim, note that each point $\bm{a}\in\partial\mathcal{T}$ is represented by two opposite vectors in $\K'_2$, say $\pm\bm{x}(\bm{a})$. This is because each face of $\partial\mathcal{T}$ is represented by two triangles in $\K'_2$, one being obtain from the other by a reflection in the origin. Therefore, from the unfolding process, one obtains that for any given points $\bm{a}, \bm{b}\in\partial\mathcal{T}$, 
\begin{equation}\label{distgeotetra}
\textrm{dist}_{\partial\mathcal{T}}\left(\bm{a}, \bm{b}\right)\;\asymp\; 2\min\left\{\left\|\frac{\bm{x}(\bm{a})-\bm{x}(\bm{b})}{2}\right\|_{\Z^2},\, \left\|\frac{\bm{x}(\bm{a})+\bm{x}(\bm{b})}{2}\right\|_{\Z^2}\right\}.
\end{equation}
Here, the left--hand side is shorthand notation for the geodesic length on $\partial\mathcal{T}$ between the points  $\bm{a}$ and $ \bm{b}$, the right--hand side stands for the minimum of the toral distances in $\K'_2$ between $\bm{x}(\bm{a})$ and the points  $\pm\bm{x}(\bm{b})$, and the implicit constant depends only on the linear invertible (hence bi--Lipschitz) map induced by the matrix $\mathfrak{M}$. This relation shows  in particular that the ``folding back'' process mapping a point in $\K'_2$ to a point on $\partial\mathcal{T}$  can only decrease geodesic distances. Therefore,  the density of the tetrahedral sequence $\left(\mathfrak{t}_{\bm{\alpha}}(p)\right)_{p>j_{k-1}}$ is a consequence of the density of  the sequence $\left(\tilde{\mathfrak{s}}_{\bm{\alpha}}(p)\right)_{p>j_{k-1}}$ in $\K'_2$ identified with the $2\times 2$ torus. 

The density of this toral sequence is  easy to establish~: since the numbers $1, \alpha_1$ and $\alpha_2$ are linearly independent over the rationals (this readily follows from the equivalent formulation of bad approximability given in~\eqref{transference}), Weyl's criterion implies that the sequence  $\left(\tilde{\mathfrak{s}}_{\bm{\alpha}}(p)\right)_{p>j_{k-1}}$ is equidistributed in $\K'_2$, whence the claim.\\

\paragraph{\emph{Step 3:  the spherical sequence $\left(\bm{u}_k\right)_{k\ge 1}$ satisfies the separation condition~\eqref{unifdiscr}.}} The previous step implies in particular that the sequence $\left(\bm{u}_k\right)_{k\ge 1}$ is well--defined provided that $\gamma>0$ is small enough.  Let then $k, m\ge 1$ be integers such that $m<k^{2/3}$. From the definition of this sequence in~\eqref{defj_ku_k}, it is immediate that $$\min_{1\le m < k^{2/3}}\; \dsph\left(\bm{u}_{k-m}, \bm{u}_{k}\right)\;>\; \frac{\gamma}{\sqrt[3]{k}}\cdotp$$ This same definition also implies that $$\min_{1\le m < k^{2/3}}\; \dsph\left(\bm{u}_{k+m}, \bm{u}_{k}\right)\;>\; \frac{\gamma}{\sqrt[3]{k+m}}\;>\; \frac{\gamma}{\sqrt[3]{k+k^{2/3}}},$$ which shows that the separation condition~\eqref{unifdiscr} indeed holds.\\

\paragraph{\emph{Step 4:  the spherical sequence $\left(\bm{u}_k\right)_{k\ge 1}$ satisfies the density condition~\eqref{reldens}.}} We first note that the toral sequence $\left(\tilde{\mathfrak{s}}_{\bm{\alpha}}(p)\right)_{p\ge 1}$ satisfies the following density condition~: there exists a constant $K>0$ such that for all $\bm{x}\in\K'_2$ and all $k\ge 2$, one can find an integer $p\ge 1$ such that $p\le Kk^{2/3}$ and 
\begin{equation}\label{auxildenss}
\left\|\frac{\tilde{\mathfrak{s}}_{\bm{\alpha}}\left(j_{k-1}+p\right)-\bm{x}}{2}\right\|_{\Z^2}\;\le\; \frac{1}{\sqrt[3]{k}}\cdotp
\end{equation}
This is indeed an immediate consequence of Proposition~\ref{propdensbad}, which implies the existence of such $K>0$ with the property that the sequence $\left(\left(\tilde{\mathfrak{s}}_{\bm{\alpha}}\left(j_{k-1}+p\right)+\bm{1}\right)/2\right)_{1\le p\le Kk^{2/3}}$ is $k^{-1/3}$---dense in $\K_2$ identified with the unit torus. 

In particular, with the notation of Step 2, if $\bm{a}\in\partial\mathcal{T}$ is such that $\bm{x}=\pm\bm{x}\left(\bm{a}\right)$ (\footnote{This just means that $\bm{x}\in\left\{ -\bm{x}\left(\bm{a}\right), \bm{x}\left(\bm{a}\right)\right\}$.}), it follows from the fact that $\tilde{\mathfrak{s}}_{\bm{\alpha}}(p)=\pm\bm{x}\left(\mathfrak{t}_{\bm{a}}(p)\right)$ for all $p\ge 1$, from~\eqref{distgeotetra} and from~\eqref{auxildenss} that 
\begin{equation}\label{toraltetradens}
\textrm{dist}_{\partial\mathcal{T}}\left(\mathfrak{t}_{\bm{\alpha}}\left(j_{k-1}+p\right), \bm{a}\right)\;\ll\; \frac{1}{\sqrt[3]{k}}\qquad \textrm{for some} \qquad p\in\llbracket 1, Kk^{2/3}\rrbracket .
\end{equation}

Let then $m\ge 0$ denote the integer such that 
\begin{equation}\label{domainer}
j_{k-1+m}< j_{k-1}+p \le j_{k+m}.
\end{equation}  
Since $\left(j_k\right)_{k\ge1}$ is a strictly increasing sequence of integers, the first inequality in~\eqref{domainer} and the second relation in~\eqref{toraltetradens} imply that 
\begin{equation}\label{bound1}
0\;\le\; m\;\ll\; k^{2/3}.
\end{equation}

Set $\bm{v}=\mathfrak{r}^{-1}\left(\bm{a}\right)\in\Sph^2$ and assume first that $$j_{k-1}+p=j_{k+m}.$$ Then, it holds that 
\begin{equation}\label{reldensfirststep}
\dsphb\left(\bm{u}_{k+m}, \bm{v}\right)\;\ll\; \frac{1}{\sqrt[3]{k}}\cdotp
\end{equation} 
This indeed follows from relations~\eqref{lipradial} and~\eqref{toraltetradens} combined with this obvious observation~: for any $\bm{a}, \bm{b}\in\partial\mathcal{T}$, 
\begin{equation}\label{euclgoedtetra}
\left\|\bm{a}-\bm{b}\right\|_2\;\le\;\textrm{dist}_{\partial\mathcal{T}}\left(\bm{a}, \bm{b}\right).
\end{equation}

Assume now that $$j_{k-1+m}< j_{k-1}+p < j_{k+m}.$$ From the definition of the integer $j_{k+m}$ in~\eqref{defj_ku_k}, the unit vector $\left(\mathfrak{r}\circ\mathfrak{t}_{\bm{\alpha}}\right) (j_{k-1}+p)$ lies in $\mathfrak{D}_{k+m}$; that is, there exists $l\in\llbracket 1, (k+m)^{2/3}\rrbracket$ such that $$\dsphb \left(\bm{u}_{k+m-l}, \; (\mathfrak{r}\circ\mathfrak{t}_{\bm{\alpha}}) (j_{k-1}+p)\right)\;\le\;\frac{\gamma}{(k+m)^{1/3}}\;\le\; \frac{\gamma}{k^{1/3}},$$ where 
\begin{equation}\label{boundm-l}
\left|m-l\right|\;\ll\; k^{2/3}.
\end{equation}

Therefore, 
\begin{align}
\dsphb \left(\bm{u}_{k+m-l}, \bm{v}\right)\; &\le\;\dsphb \left(\bm{u}_{k+m-l}, (\mathfrak{r}\circ\mathfrak{t}_{\bm{\alpha}}) (j_{k-1}+p)\right)+\dsphb \left((\mathfrak{r}\circ\mathfrak{t}_{\bm{\alpha}}) (j_{k-1}+p), \bm{v}\right)\nonumber\\
&\ll\; \frac{1}{\sqrt[3]{k}}\cdotp \label{reldensbistetra}
\end{align}
Here, the last relation relies again on~\eqref{lipradial}, \eqref{toraltetradens} and~\eqref{euclgoedtetra}.

Inequalities~\eqref{bound1} and~\eqref{reldensfirststep} on the one hand and~\eqref{boundm-l} and~\eqref{reldensbistetra} on the other thus establish that the density condition~\eqref{reldens} holds for the spherical sequence $\left(\bm{u}_k\right)_{k\ge 1}$. This completes Step 4 and the proof of Theorem~\ref{thmdistrtetra}.
\end{proof}

\section{Some Open Problems}\label{oppb}

In this section, we collect some open questions emerging from the theories developed in this paper.\\

\paragraph{\textbf{Distribution of Spherical Flows obtained from Regular Polyhedra in $\R^3$}} The analysis of the spherical sequence constructed from a tetrahedral flow in Section~\ref{sec3} relied on a fondamental property of the tetrahedron~: its net can be arranged so as to tile the plane in a way that is compatible with the identification of the edges. This enabled us to reduce the distribution properties of the corresponding discrete flow (that is, its density, its separation and its equidistribution) to that of a toral flow.

It seems natural to extend this study to similar spherical sequences constructed from flows on the surfaces of the other Platonic solids (cube, octahedron, dodecahedron and icosahedron). The problem is then to determine Diophantine conditions on the direction of the flow so that the resulting discrete sequence should be equidistributed, dense and/or well--separated on the surface of the solid (with effective bounds on its discrepancy and dispersion, which are natural measures of equidistribution and density, respectively --- see~\cite[\S~1.1.3]{drmtich} for further details). 

At this level of generality, the situation is nevertheless more difficult to deal with than in the case of the tetrahedron, as the net of a Platonic solid does not necessary tile the plane in a way that is compatible with the identification of the edges. This is already the case for the cube, and a consequence of this phenomenon is illustrated in Figure~\ref{singularitycube}~: vertices are singularities for linear flows in the sense that two initially close flows can eventually follow two very different paths when they are split by a vertex. We expect that the distribution properties of a discrete flow on the surface of a cube will require the consideration of  Diophantine properties which do not reduce to bad approximability alone.

\begin{figure}[h!]

\centering
\begin{tikzpicture}[scale=2,tdplot_main_coords]
\draw[thick] (2,0,0)--(2,2,0);
\draw[thick] (2,2,0)--(0,2,0);
\draw[thick] (0,0,2)--(0,2,2);
\draw[thick] (0,2,0)--(0,2,2);
\draw[thick] (2,0,2)--(2,2,2);
\draw[thick] (2,2,2)--(2,2,0);
\draw[thick] (2,2,2)--(0,2,2);
\draw[thick] (2,0,2)--(2,0,0);
\draw[thick] (2,0,2)--(0,0,2);

\draw (2,0.95,1.05)--(2,1.9,2);
\draw[->] (2,0,0.1)--(2,0.95,1.05);

\draw (2,1.9,2)--(1.9,2,2);

\draw (0.95,2,1.05)--(0,2,0.1);
\draw[->] (1.9,2,2)--(0.95,2,1.05);

\draw (2,1.05,0.95)--(2,2,1.9);
\draw[->] (2,0.1,0)--(2,1.05,0.95);

\draw (1.95,2,1.95)--(1.9,2,2);
\draw[->] (2,2,1.9)--(1.95,2,1.95);

\draw (0.95,1.05,2)--(0,0.1,2);
\draw[->] (1.9,2,2)--(0.95,1.05,2);

\end{tikzpicture}
\caption{Unstability of a linear flow on the surface of a cube~: two lines translated by a small amount from each other can determine two drastically different paths when they are split by a vertex of the cube.}
\label{singularitycube}
\end{figure}
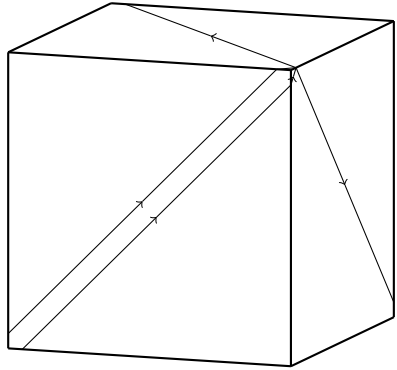

It should be pointed out that the recent paper~\cite{plato} provides a nice study of some of the topological properties of the surfaces of the Platonic solids.\\

\paragraph{\textbf{Analogous Spherical Flows in Higher Dimensions}} A natural question extending the previous problem is the study of the distribution of linear flows on the $(n-1)$--dimensional surface of a regular convex polytope in any dimension $n\ge 3$. In this respect, recall that the tetrahedron (that is, a 3--simplex), the cube and the octahedron  are the only three--dimensional convex regular solids which admit generalisations to any dimensions (in the form of the $n$--dimensional simplex,  hypercube and  hyperoctahedron, respectively). 

In this case also, the study of the distribution of linear flows on the surfaces of simplices in dimension $n\ge 4$ will require arguments different from those presented in Section~\ref{sec3}. Indeed, it can be shown that whenever $n\ge 4$, the net of an $n$--dimensional simplex does not tile the $(n-1)$--dimensional space in a way that is compatible with the identification of the edges.\\

\paragraph{\textbf{Spiral Delone Sets admitting an Asymptotic Density}} A measure of the degree of regularity of a sequence of points $\mathcal{S}=\left(\bm{s}_k\right)_{k\ge 1}$ in $\R^n$ is its asymptotic density, whenever it exists. Recall its definition~: $\mathcal{S}$ is said to have asymptotic density $r>0$ if for any bounded measurable set $X\subset\R^n$ whose boundary has zero Lebesgue measure, it holds that $$\frac{\#\left\{k\ge 1\;:\; \bm{s}_k\in RX \right\}}{R^n}\;\underset{R\rightarrow\infty}{\longrightarrow}\; r\cdot \lambda_n(X).$$

It follows from an elementary adaptation of the argument presented in~\cite[Proposition~1]{marklof} that the spiral set defined in~\eqref{sk} has asymptotic density $r=V_n^{-1}$, where $V_n$ denotes the volume of the $n$--dimensional unit ball, if, and only if, the corresponding spherical sequence $\left(\bm{u}_k\right)_{k \ge 1}$ is uniformly distributed in the following sense~: for any $\bm{w}\in\Sph^d$ and any $\rho>0$, $$\frac{\#\left\{1\le k\le N\;:\; \bm{u}_k\in \mathcal{C}_d\left(\bm{w}, \rho\right) \right\}}{N}\;\underset{N\rightarrow\infty}{\longrightarrow}\; \frac{\sigma_d\left(\mathcal{C}_d\left(\bm{w}, \rho\right) \right)}{\sigma_d\left(\Sph^d\right)}\cdotp$$
Uniform distribution is here defined in terms of the spherical \emph{measure} $\sigma_d$, and is a property of regularity rather different from the density and separation conditions introduced in Propositions~\ref{cnsreldens} and~\ref{cnsunifdiscr}. The latter two conditions are indeed defined in terms of the geodesic  \emph{distance} $\dsph$. 

It is not clear to the authors whether a spiral Delone set with asymptotic density $r=V_n^{-1}$ can exists in any dimension $n\ge 3$.\\


\end{document}